\documentclass{iopart}
\usepackage{iopams}
\usepackage{setstack}
\usepackage{latexsym} 
\usepackage{mathrsfs}
\usepackage{amsthm} 
\usepackage[dvips]{graphicx}

 \usepackage{color}
\usepackage{cases}
\usepackage[colorlinks=true,pdfpagemode=FullScreen]{hyperref}
\usepackage{verbatim}
\usepackage[toc,page]{appendix}
\renewcommand{\theequation}{\thesection.\arabic{equation}}

\newenvironment{draft}{\color{red}}{}
\newtheorem{thm}{Theorem}[section]
\newtheorem{cor}[thm]{Corollary}
\newtheorem{lem}[thm]{Lemma}
\newtheorem{prop}[thm]{Proposition}

\theoremstyle{definition}
\newtheorem{defn}[thm]{Definition}
\newtheorem{rem}[thm]{Remark}

\newtheorem{ex}[thm]{Example}

\newcommand{\N}{{\mathbb N}}

\newcommand{\R}{{\mathbb R}}
\newcommand{\T}{{\mathbb T}}
\newcommand{\conv}{\mathop{\mathrm{co}}}
\newcommand{\dom}{\mathop{\mathrm{dom}}}

\newcommand{\Lip}{\mathop{\mathrm{Lip}}}

\def\bc{\begin{center}}       \def\ec{\end{center}}
\def\be{\begin{equation}}     \def\ee{\end{equation}}
\def\ba{\begin{array}}        \def\ea{\end{array}}
\def\bea{\begin{eqnarray}}            \def\eea{\end{eqnarray}}
\def\beaa{\begin{eqnarray*}}  \def\eeaa{\end{eqnarray*}}
\def\bd{\begin{draft} }      \def\ed{\end{draft}}

\def\p{\partial}

\def\F1{\mathcal{F}}

\def\sig{\sigma}
\def\ep{\epsilon}

\def\x{x_0}
\def\y{y_0}

\def\sig{\sigma}

\def\p{\partial}

\def\l{\langle}
\def\r{\rangle}

\def\b{\bar{R}}
\def\bx{\bar{x}}
\def\ep{\epsilon}

\def\vp{\varphi}

\begin{document}
\title[Viscosity solution of Hamilton-Jacobi equation by a minimax method]{Viscosity solution of Hamilton-Jacobi equation by a limiting minimax method}
\author{Qiaoling WEI}
\address{WEI Qiaoling, Universit\'{e} Paris 7, UFR de Math\'ematiques, B\^atiment Sophie Germain
5, rue Thomas Mann, 75205 Paris Cedex 13
FRANCE} 

\ead{weiqiaoling@math.jussieu.fr}


{ 
\begin{abstract}

For non convex Hamiltonians, the viscosity solution and the more geometric minimax solution of the Hamilton-Jacobi equation do not coincide in general.
They are nevertheless related: we show that iterating the minimax procedure during shorter and shorter time intervals one recovers
the viscosity solution.

\end{abstract}

\noindent{\it Keywords\/}: Hamilton-Jacobi equation, viscosity solution, minmax solution, geometric solution, generating family
\ams{35D40,35F21,37J05,49L25}
\maketitle

\section{Introduction}
Consider the Cauchy problem for the Hamilton-Jacobi equation
\[\mbox{(H-J)}\quad \quad\quad \left\{
    \begin{array}{ll}
      \p_t u(t,x)+ H(t,x, \p_x u) = 0,  \\
      u(0,x)= v(x)
    \end{array}
  \right.
  \]
By the classical characteristics method (ref. section \ref{charat}), one can to some extent solve this problem  when $v$ is $C^2$: solving Hamilton's equations $\dot x=\p_pH(t,x,p)$ and $\dot p=-\p_xH(t,x,p)$, one gets the characteristic lines $\big(t,x(t),p(t)\big)$\/, from which the solution $u(t,x)$ of the Cauchy problem can be obtained as follows: setting $u_t(x)=u(t,x)$\/, the graph of  $du_t$ is the section at time $t$ of the union $L$ of the  characteristic lines passing through $\big\{\big(0,dv(x)\big)\}$\/, whereas $\p_t u$ is provided by the equation. 

When $H$ is not linear with respect to $p$\/, this procedure does not yield a global solution of the problem in general, as the geometric solution $L$ is not the graph of a function $p(t,x)$\/. In other words the \emph{wavefront} $\mathcal{F}$\/, obtained in $(t,x,u)$ space by solving the equation $du=-H(t,x,p)\,dt+p\,dx$ restricted to $(t,x,p)\in L$\/---and taking into account the values of $v$ and not only $dv$---is not the graph of a function (which otherwise would be a solution of the problem): the projected characteristics $\big(t,x(t)\big)$ may have a variation of slopes and cross after some time. 

In some applications, e.g. to geometrical optics, the wavefront $\mathcal{F}$ provides a solution of the physical problem at hand. In other cases, however, one does look for a single-valued solution $u(t,x)$\/, which can \emph{not} be $C^2$ since, otherwise, its graph would be $F$\/. Therefore, some new idea is needed to obtain a weak solution in the sense of distributions.

The most natural idea, assuming that the projection of $F$ into $(t,x)$ space is onto, is to get such a solution as a section of the wavefront, obtained by selecting a single $u$ over each $(t,x)$. When the Hamiltonian is convex enough with respect to $p$ (and $v$ is not too wild at infinity), such a ``graph selector'' consists in choosing the lowest point of $\mathcal{F}$ above each $(t,x)$\/, i.e. taking for $u(t,x)$ the smallest $u$ with $(t,x,u)\in \mathcal{F}$\/. Analytically, this is expressed by  formulae due to Hopf-Lax \cite{Hopf} for Hamiltonians of the form $H(p)$ and to Lax-Oleinik in the general case \cite{Fathi}; these formulae, which extend to non-differentiable initial data $v$, are widely used in  weak KAM theory; they imply that the min selector provides a Lipschitz weak solution of the problem.

This min solution turns out to be the \emph{viscosity solution}, which owes its name to the fact that it was first introduced as the limit when $\varepsilon\to 0^+$ of the solution of the same Cauchy problem for the equation $\p_t u(t,x)+ H(t,x, \p_x u) =\varepsilon\Delta_xu(t,x)$\/; the much more tractable general definition\footnote{For conservation laws, more particular characterizations called entropy conditions had been obtained by Oleinik  in dimension one \cite{Lax} and by Kruzkov in arbitrary dimensions \cite{Kru}.} was given in 1983 by M.G.~Crandall, L.C.~Evans and P.L.~Lions for general nonlinear first order partial differential equations \cite{PL,GL}. 

In the non-convex case, it has been known for quite a long time that the viscosity solution, when it exists, may \emph{not} be a section of the wavefront (see for example \cite{Alain}). On the other hand, following a remark of J.-C. Sikorav, M.~Chaperon introduced in 1991 \cite{loi} weak solutions whose graph is a section of the wavefront, obtained by a ``minmax'' procedure which generalises the minimum considered in the convex case and relies on the existence of suitable generating families for the geometric solution---in addition to \cite{Alain}, see \cite{Vi2,AC,Cardin} for examples showing that  minmax and viscosity solutions may differ.

\bigskip
The aim of the present paper is to relate such minmax solutions to the viscosity solutions, which possess very general existence, uniqueness and stability properties in a large class of ``good'' cases but are not so easy to understand geometrically in general, one of the motivations for this work.

Among the  properties of viscosity solutions \emph{not} shared by minmax solutions---see \ref{rarefaction}---is the ``semigroup'' (Markov) property. Hence, it is natural to try and get this semigroup as a limit of (discrete) semigroups obtained by dividing a given time interval into small pieces and taking the minimax step by step (``iterated minimax procedure''). Our main purpose is to show that when the small time intervals go to zero, one  gets a limit solution which is indeed the viscosity solution (see Theorem \ref{main}). This answers a question of M.~Chaperon. 

\ack I am grateful to Marc Chaperon, who introduced me to this subject; I owe  him the conjecture of iterated minmax and many useful discussions. I also thank Alain Chenciner for his help and discussions during the work. And I wish to thank  Patrick Bernard for motivating the final proof of the main theorem, which simplifies the original one. After the first version of this article was posted on arXiv.org, Professor P.E.~Souganidis kindly informed the author about his article \cite{Soug} on the approximation schemes to viscosity solutions. It turns out that the minmax selector behaves like a ``generator'' defined in \cite{Soug}, and our limiting iterated minmax procedure fits into his general approximation schemes. 

 
\section{Generating families and minmax selector}
%
Consider the cotangent bundle $T^*M$ of a connected manifold $M$ of dimension $m$, endowed with its canonical symplectic form $\omega=dp\wedge dq$. A Lagrangian submanifold $L$ is a submanifold of dimension $m$ such that $\omega$ vanishes on $L$.

\begin{defn}
A \emph{generating family} for a Lagrangian submanifold $L\subset T^*M$ is
a $C^2$ function $S: M\times \R^k\to \R$ such that  $0$ is a regular value of the map $(x,\eta)\mapsto \p S(x,\eta)/\p \eta$ and
\[
L=\Big\{\big(x,\frac{\p S}{\p x}(x,\eta)\big):\frac{\p S}{\p \eta}=0\Big\}\/;
\]
more precisely, the condition that $0$ is a regular value implies
that the {\it critical locus} $\Sigma_S:=\{(x,\eta)|\p_{\eta}S=0\}$
is a submanifold and that the map
\[
i_S:\Sigma_S\to T^*M,\quad (x,\eta)\mapsto (x,\p_{x}S(x,\eta))
\]
is an immersion; we require that $i_S$ be an embedding and, of course, $i_S(\Sigma_S)=L$.
\end{defn}

A function $S$ on $M\times \R^k$ need not have  critical points.
However, it does have critical points if we prescribe some behavior
at infinity as in the following definition:

\begin{defn}\label{gfc}
A generating family $S: M\times \R^k\to \R$ of $L$ is \emph{quadratic at infinity} if there exists a \emph{nondegenerate} quadratic form $Q$ such that, for any compact subset $K\subset M$, the differential $|\p_{\eta}(S(x,\eta)- Q(\eta))|$ is bounded on $K\times \R^k$.
\end{defn}
We will abbreviate ``generating family quadratic at infinity'' by G.F.Q.I.. Remark that, the $C^2$ G.F.Q.I. such defined can be made into ``exactly quadratic at infinity'': for every compact subset $K\subset M$, $(S-Q)|_{K\times \R^k}$ has compact support modulo a fiberwise diffeomorphism. See for example\cite{theret2}, or \cite{these} Proposition 1.21. 

%
%
%
%
%
 
Consider the sub-level sets $ S_{x}^a:=\{\eta: S(x,\eta)\leq a\},$
 the homotopy type of $(S_{x}^a,S_{x}^{-a})$ does not depend on $a$ and coincides with the homotopy type of $(Q^a,Q^{-a})$ when $a$ is large enough, we may write it as
 $(S_{x}^{\infty},S_{x}^{-\infty})$. If the Morse index of $Q$ is
 $k_{\infty}$, then
 \[
 H_i (S_{x}^{\infty},S_{x}^{-\infty};\mathbb{Z}_2)=
 H_{i}(Q^{\infty},Q^{-\infty};\mathbb{Z}_2)\simeq
                   \left\{
                                        \begin{array}{ll}
                                          \mathbb{Z}_2, & i=k_{\infty} \\
                                          0, & \mbox{otherwise}
                                        \end{array}
                                      \right.
 \]
 \begin{defn} The {\it minmax} function is defined as
 \[
 R_S(x):= \inf_{[\sigma]=A}\max_{\eta \in |\sigma|} S(x,\eta)\/,
 \] 
 where $A$ is a generator of the homology group $H_{k_{\infty}}(S_{x}^{\infty},S_{x}^{-\infty};\mathbb{Z}_2)$ and $|\sigma|$ denotes the image of the relative singular homology cycle $\sigma$\/. Such a cycle $\sigma$ will be called a \emph{descending cycle.}
 
 \emph{The function $R_S$ is determined} (up to the addition of a constant) \emph{by $L$ and does not depend on the particular choice of its G.F.Q.I.  $S$\/}: indeed, by a  theorem of Viterbo and Th\'eret \cite{theret}, for a given
 Lagrangian submanifold $L$, $S$ is  unique up to the following three operations:
 
 (a) Fiberwise diffeomorphism :
 $\tilde{S}(x,\eta):=S(x,\varphi(x,\eta))$, where $(x,\eta)\mapsto
 (x,\varphi(x,\eta))$ is a fiberwise diffeomorphism.
 
 (b) Addition of a constant: $\tilde{S}(x,\eta):=S(x,\eta)+C$.
 
 (c) Stabilization: $\tilde{S}(x,\eta,\xi):= S(x,\eta)+ q(\xi)$,
 where $q$ is a nondegenerate quadratic form.
 \end{defn}
 It is well known that the minmax $R_S(x)$ is a critical value of $S(x,\cdot)$ for every $x$\/,\footnote{Hence, the graph of $R_S$ is a section of the ``wavefront'' obtained in $(x,u)$-space (up to vertical translation) by solving the equation $du=p\,dx$ restricted to $L$\/.} a property which we shall establish in a  more general Lipschitz framework (Propostion~\ref{propun44}). 
 


%
%
%

 In the rest of the paper, we will take the manifold $M$ to be
 $\R^d$, in which case the generating families are constructed
 explicitly. The case where $M=\T^d$ can be treated via the same construction. For a general manifold, one can embed it into some
 $\R^d$ and use the trick of Chekanov \cite{Chek,Bru} to get generating families
 from those in $\R^d$\/.
 
 The contents of this section are organized as follows: first, we present an explicit formula for generating families, due to M.~Chaperon and J.-C.~Sikorav; then the notion of ``quadratic at infinity'' is discussed; finally, we will arrive at a generalization of the G.F.Q.I.'s to Lipschitz cases.
 
 \subsection{Construction of generating functions and families}
 In the following, we equip $\R^k$ with the Euclidien $\ell^2$ norm $|\cdot|$, and matrices in $\R^k$ with the associated operator norm. We denote by $\Lip(f)$ the Lipschitz constant of a function $f$ and by $\pi:T^*\R^d\rightarrow\R^d$ the canonical projection $\pi(x,y)=x$.
 
 We denote by $H:[0,T]\times T^*\R^d\rightarrow\R$  a $C^2$ Hamiltonian satisfying
 \be
 \label{equnun}
 c_H:=\sup|D^2H_t(x,y)|<\infty
 \ee
 and by $X_{H_t}$ the associated time-depending Hamiltonian vector field\footnote{We use the convention of sign that $X_H=(\p_pH,-\p_q H)$.}. By the general theory of  differential equations, as $c_H=\max_t\Lip(DH_t)=\max_t\Lip(X_{H_t})$\/, the Hamiltonian transformation $\vp_H^{s,t}$  obtained by integrating $X_{H_\tau}$ from $\tau=s$ to $\tau=t$ is a well-defined diffeomorphism for all $(s,t)\in[0,T]$. For simplicity, we  sometimes write $\vp_s^t=(X_s^t,Y_s^t):=\vp_H^{s,t}$ without mentioning $H$.
 
 We will be mostly interested in the special case where $H$ has compact support, and consider the Lagrangian submanifolds of $T^*\R^d$ which are Hamiltonianly isotopic to the zero section:
 \[
 \mathcal{L}:=\{L=\vp(dv),\quad v\in C^2\cap C^{\Lip}(\R^d),\,\vp\in
 Ham_c(T^*\R^d)\}\/;
 \]
 here $C^{\Lip}(\R^d)$ denotes the space of globally Lipschitz functions and
\[\eqalign{
 dv:=\{(x,dv(x)),\,x\in \R^d\}\subset T^*\R^d\\
 Ham_c(T^*\R^d)=\{\vp=\vp_H,\, H\in C_c^2([0,1]\times
 T^*\R^d)\}}
\]
 where $\vp_H=\vp_H^{0,1}$ is the endpoint of the isotopy (``Hamiltonian flow'') defined by $H$.
 
 \begin{defn}
 A diffeomorphism $\vp :T^*\R^d\to T^*\R^d$ admits a {\it generating function} $\phi$, if $\phi: T^*\R^d\to \R$ is of class $C^2$, such that $((x,y),(X,Y))\in \hbox{Graph}(\vp)$ if and only if
 \[
\left\{
\begin{array}{ll}
x=X+\p_y \phi(X,y),\\
Y=y+\p_X \phi (X,y).\end{array}
\right.
 \]
 \end{defn}
 
 This definition can be interpreted as follows:  the isomorphism

\[ 
\eqalign{ I:\quad T^*\R^d\times T^*\R^d\to T^*(T^*\R^d)\\
        (x,y,X,Y)\mapsto (X,y,Y-y,x-X)}
        \]
is symplectic if $T^*\R^d$ is equipped with the standard symplectic form
 $\omega=dx\wedge dy$ and $T^*\R^d\times T^*\R^d$ with the symplectic form $(-\omega)\oplus\omega=dX\wedge dY-dx\wedge dy$\/; this symplectic isomorphism $I$ maps the diagonal of the space $T^*\R^d\times T^*\R^d$ to the zero section of the cotangent space
 $T^*(T^*\R^d)$ and $\hbox{Graph}(\vp)$ to $\hbox{Graph}(d\phi)$.
 
 Hence, if it exists, the generating function $\phi$ is unique up to the addition of a constant.
 
  
   
 \begin{lem}[\cite{FG}]
  \label{1.1}
If $\delta_H:= c_H^{-1}\ln2$, then, for  $|s-t|<\delta_H$\/, the map
    \[
    \alpha_s^t: (x,y) \mapsto \big(X_s^t(x,y),y\big)
    \]
    is a diffeomorphism. As a consequence, $\vp_s^t$ admits the generating function
 \be\label{gfd}
 \phi_s^t(X,y) = \int_s^t \Big((Y_s^{\tau}-y)\dot{X_s^{\tau}}-H(\tau,X_s^{\tau},Y_s^{\tau})\Big)d\tau
 \ee  where $(X_s^{\tau}(X,y),Y_s^{\tau}(X,y))=\vp_s^{\tau}\circ(\alpha_s^t)^{-1}(X,y)$ and the dot denotes the derivative with respect to $\tau$.
 \end{lem}
 
 
\begin{lem} \label{dergf}For the generating function $\phi_s^t$ defined by (\ref{gfd}), we have
\[
\p_s\phi_s^t(X,y)=H(s,x,y),\quad \p_t\phi_s^t(X,y)=-H(t,X,Y)
\]
where $(X,Y)=\vp_s^t(x,y)$.
\end{lem}

\begin{proof} 
Differentiating (\ref{gfd}) on both sides, we have

\beaa\fl
 \p_s \phi_s^t(X,y)=H(s,x,y)+\int_s^t\big(\frac{d}{ds}Y_s^{\tau}\frac{d}{d\tau}X_s^{\tau}+(Y_s^{\tau}-y)\frac{d}{ds}\frac{d}{d\tau}X_s^{\tau}+
\frac{d}{d\tau}Y_s^{\tau}\frac{d}{ds}X_s^{\tau}-\frac{d}{d\tau}X_s^{\tau}\frac{d}{ds}Y_s^{\tau}\big)d\tau\\
=H(s,x,y)+\int_s^t(Y_s^{\tau}-y)\frac{d}{ds}\frac{d}{d\tau}X_s^{\tau}d\tau+Y_s^{\tau}\frac{d}{ds}X^{\tau}_s|_s^t-\int_s^tY_s^{\tau}\frac{d}{d\tau}\frac{d}{ds}X_s^{\tau}d\tau\\
= H(s,x,y)\eeaa
where we have used $\p_1H_{\tau}(X_s^{\tau},Y_s^{\tau})=-\dot{Y}_s^{\tau}$, $\p_2H_{\tau}(X_s^{\tau},Y_s^{\tau})=\dot{X}_s^{\tau}$, and $X_s^t\equiv X$. Similarly, we have
\[ 
\p_t\phi_s^t(X,y)=-H(t,X,Y).
\]
\end{proof}

 \begin{prop}[Composition formula \cite{sikorav}]
 \label{comf}
 If a Lagrangian submanifold $L_0\subset T^*\R^d$ admits a
 generating family
  $S_0:\R^d\times \R^k\to \R$, then for $|t-s|<\delta_H$, the Lagrangian submanifold $\vp_s^t(L_0)$ has the generating family
  \be
  \label{ind}
 S(x,(\xi,\x,\y))=S_0(\x,\xi)+\phi_s^t(x,\y)+xy_0-\x\y
 \ee
 \end{prop}
 
The proof is straightforward.
 
 \begin{cor}
 \label{comp3}
 For each subdivision $0\leq s=t_0<t_1\dots<t_N=t\leq T$ satisfying $|t_i-t_{i+1}|<\delta_H$\/, if $\phi_H^{t_i,t_{i+1}}$ is the generating function of $\vp_H^{t_i,t_{i+1}}$ defined in Lemma~\ref{1.1}, we have the following for each $C^2$ function $v:\R^d\rightarrow\R$\/:
 \begin{enumerate}
 \item[i)]
 A generating family $S:\R^d\times (T^*\R^d)^N\to \R$ of the Lagrangian submanifold $\vp_H^{s,t}(dv)$ is
 \be
 \label{nnn}
 S(x,\eta)= v(x_0)+ \sum_{0\leq i<N} \phi_H^{t_i,t_{i+1}}(x_{i+1},y_i)+ \sum_{0\leq i<N}(x_{i+1}-x_i)y_i\/,
 \ee
 where $x_N:=x$, $\eta=\big((x_i,y_i)\big)_{0\leq i<N}$\/.
 \item[ii)]
One defines a $C^2$ family $S:[s,t]\times \R^d\times (T^*\R^d)^N\to\R$ such that each $S_\tau:=S(\tau,\cdot)$ is a generating family for $\vp_H^{s,\tau}(dv)$ as follows: let $\tau_j=s+(\tau-s)\frac{t_j-s}{t-s}$,
\be\label{gfqi}
S(\tau,x,\eta)=v(x_0)+\sum_{0\leq i<N}\phi_H^{\tau_i,\tau_{i+1}}(x_{i+1},y_i)+\sum_{0\leq i<N}(x_{i+1}-x_i)y_i\ee
 \item[iii)]
 For each critical point $\eta$ of $S(\tau,x;\cdot)$\/, the corresponding critical value is
 \[
 S_\tau(x;\eta)=v(x_0)+\int_s^\tau \left(Y_s^{\sigma}\dot{X_s^{\sigma}}-H\big(\sigma,X_s^{\sigma},Y_s^{\sigma}\big)\right)d\sigma\/,
 \]
 where $X_s^{\sigma}:=X_s^{\sigma}\big(x_0,dv(x_0)\big)$ and $Y_s^{\sigma}:=Y_s^{\sigma}\big(x_0,dv(x_0)\big)$\/. Hence, the critical values of $S(\tau,x;\cdot)$ are the real numbers
 \be
 \label{equncinq}
 v\big(X_\tau^s(z)\big)+\int_s^\tau \left(Y_\tau^{\sigma}(z)\dot{X_\tau^{\sigma}}(z)-H\big(\sigma,X_\tau^{\sigma}(z),Y_\tau^{\sigma}(z)\big)\right)d\sigma
 \ee
 with $z:=(x,y)$\/, $y\in\pi^{-1}(x)\cap\vp_H^{s,\tau}(dv)$\/.
 \end{enumerate}
 \end{cor}
 
 \begin{proof}
 i) As the Hamiltonian flow is a ``two-parameter groupoid'', we have that
 \[
 \vp_H^{s,t}=\vp_H^{t_0,t_N}=\vp_H^{t_{N-1},t_N}\circ\dots \circ\vp_H^{t_0,t_1}\/;
 \]
 hence, if $|t_{i+1}-t_i|<\delta_H$ for all $i$, it follows from the composition formula  in Proposition~\ref{comf} that formula \ref{nnn} does define a generating family for $\vp_H^{s,t}(dv)$\/.
 
 ii) is clear.
 
 iii) is proved by inspection (and very important).
 \end{proof}
 
 \subsection{Generating families quadratic at infinity}
 
  \begin{lem}
  If $H$ has compact support, the generating families constructed in Corollary~\ref{comp3} are quadratic at infinity when the $C^2$ function $v$ is Lipschitzian.
  \end{lem}
  
  \begin{proof}
  Each $\phi_H^{t_i,t_{i+1}}$ has compact support and therefore bounded derivatives. Hence $S(x,\eta)=\psi(x,\eta)+Q(\eta)$ is quadratic at infinity with
 $Q(\eta):=-x_{N-1}y_{N-1}+ \sum_{0\leq
   i<N-1}(x_{i+1}-x_i)y_i$ and $\psi(x;\eta)=v(x_0)+xy_{N-1}+\sum_{0\leq i<N}\phi_H^{t_i,t_{i+1}}(x_{i+1},y_i)$ having bounded derivatives with respect to $\eta$ for $x$ on compact subsets of $\R^d$.
  \end{proof}

Denote $\pi$ the projection $T^*\R^d\to \R^d$. A necessary condition for $L$ to
 admit a G.F.Q.I. is that, for any compact subset $K$ of $\R^d$\/, the intersection $L\cap\pi^{-1}(K)$ be compact and \emph{nonempty}: indeed, a function on $\R^k$ whose differential, up to a bounded map, equals that of a nondegenerate quadratic form off a compact set must have critical points.
  
It follows that there does not always exist a G.F.Q.I. for $L=\vp_H^{s,t}(dv)$ if $H$ is not compactly supported, even when it satisfies (\ref{equnun}) and $v$ has as little growth at infinity as possible:
 
 \begin{ex}
 If the Hamiltonian $H\in C^2(\R\times T^*\R)$ is given by $H(t,x,y)=x^2+y^2$\/, then $\vp_H^{0,t}(x,y)=(x \cos{2 t}-y\sin{2 t},y\cos{2t}+x\sin{2t})$\/; if $v=0$\/, it follows that
 \[
 L:=\vp_H^{0,\pi/4}(dv)=\{0\}\times\R.
 \]
 has empty intersection with $\pi^{-1}(x)=\{x\}\times\R$ for $x\neq0$ and noncompact intersection with $\pi^{-1}(0)$\/, which prevents $L$ from admitting a G.F.Q.I..
 \end{ex}

%
 It is also essential to require that $v$ has little growth at infinity.
 
 \begin{ex}
 It is essential that $v$ be Lipschitzian: indeed, if $d=1$\/, $H(t,x,y)=h(x,y)=\frac{1}{2}y^2$ for $|y|\leq 1$, $|\p_x h|,|\p_y h|\leq 1/2$ and $h$ vanishes off a compact set of $\R^2$. If $v(x)=\frac{1}{3}x^3$\/, then for $t=1$, \[
 \pi\circ\vp_H^{0,1}(dv)\subset\{x+x^2,|x|\leq 1\}\cup [\frac{1}{2},+\infty]\cup[-\infty,-\frac{1}{2}.]
 \] One sees that the image under the projection $\pi$ of $L=\vp_H^{0,1}(dv)$ is not $\R$.
 \end{ex}

 As the main ingredient in the construction of generating families is the Hamiltonian flow, what matters essentially over a given compact subset of $\R^d$ is the region swept by the Hamiltonian flow; this is the idea of what is called the \emph{property of finite propagation speed} in \cite{VC}, Appendix~A:
 
 \begin{prop}
 \label{compact}
 Let $[s,t]\subset[0,T]$ and $L=\vp_H^{s,t}(dv)$.  If for any compact subset $K\subset\R^d$, the set
 \[
 \mathcal {U}_K:=\bigcup_{\tau\in [s,t]}\{\tau\}\times
 \left\{\vp_H^{s,\tau}\left(\vp_H^{t,s}\big(\pi^{-1}(K)\big)\cap dv\right)\right\},
 \]
 is non empty and compact, then $L$ admits G.F.Q.I.'s in the sense that each $L|_K:=L\cap \pi^{-1}(K)$ has a G.F.Q.I..
 \end{prop}
 
 \begin{proof}
 For any $K$, let $\tilde{H}=\chi H$, where $\chi$ is a compactly supported smooth function on $[0,T]\times T^*\R^d$ equal to $1$ in a neighbourhood of $\mathcal{U}_K$\/. Then formula \eref{nnn} with $H:=\tilde{H}$ gives a G.F.Q.I. $S_{\tilde H}$ for $L|_K=\vp_H^{s,t}\big(\pi^{-1}(K)\cap dv\big)$.
 \end{proof}
 
 \begin{rem}
 \label{localv}
 One can also truncate $v$, as the effective region for $v$ is $\pi\big(\vp_H^{t,s}(\pi^{-1}(K)\big)$. This may help to localize the minmax.
 
 Condition~\eref{equnun} is not required here, provided $H$ is $C^2$ and such that $\vp_H^{s,t}$ is defined for all $s,t\in[0,T]$\/.
 \end{rem}
 
 
 \begin{lem}
 \label{coi}
 If two families $S$ and $S'$ are quadratic at infinity with $|S-S'|_{C^0}<\infty$\/, then the associated minimax functions satisfy
 \[
 |R_{S}(x)-R_{S'}(x)|\leq |S-S'|_{C^0}\/.
 \]
 \end{lem}
 
 \begin{proof}
 If $S\leq S'$, then by definition $R_{S}(x)\leq R_{S'}(x)$. Hence, in general, the inequality $S\leq S'+ |S-S'|_{C^0}$ yields $R_S(x)\leq R_{S'}(x)+ |S-S'|_{C^0}$. We conclude by exchanging $S$ and $S'$.
 \end{proof}
 
 \begin{prop}
 Under the hypotheses of Proposition~\ref{compact} and with the notation of its proof, the Lagrangian submanifold $L$ determines a minmax function, given by
 \[
 R(x)=\inf\max S_{\tilde H}(x,\eta),\quad \hbox{if}\,\,x\in K\subset \R^d
 \]
 and independent of the truncation $\tilde H$ and the subdivision of $[s,t]$ used to define $S_{\tilde H}$\/.
 \end{prop}
 
 \begin{proof} Let $\tilde{H}$ and $\tilde{H}'$ be two truncations for $H$ on $\mathcal{U}_K$ as in the proof of Proposition~\ref{compact}. Let $H^{\mu}=\mu\tilde{H}+(1-\mu)\tilde{H'}$, $\mu\in [0,1]$\/; as the constant $c_{H^\mu}$ of \eref{equnun} is uniformly bounded, one can find a subdivision $s=t_0<t_1\dots<t_N=t$ satisfying $|t_i-t_{i+1}|<\delta_{H_\mu}$ for all $\mu$\ (see Lemma \ref{1.1}); if $S_{\mu}$ denotes the corresponding G.F.Q.I. of $L|_K=L\cap \pi^{-1}(K)$ for $0\leq\mu\leq1$ then, by Lemma \ref{coi}, as
 $S_\mu$ depends continuously on $\mu$, so does the minmax  $R_{S_{\mu}}(x)$ for $x\in\pi(L)$\/.
 
 On the other hand, $R_{S_{\mu}}(x)$ is a critical value of the map $\eta\mapsto S_{\mu}(x,\eta)$, and, by \eref{equncinq}, \emph{the set of all such critical values is independent of $\mu$ and the subdivision, and depends only on $\mathcal{U}_K$}; as it has measure zero by Sard's Theorem,  $R_{S_{\mu}}(x)$ is constant for $\mu\in [0,1]$.
 
 \goodbreak
 The fact that the critical value $R_{S}(x)$ itself does not depend on the subdivision is established in Lemma~\ref{ides}.
 \end{proof}
 
 \begin{ex} If the base manifold is $M=\mathbb{T}^d$, taking its universal
 covering $\R^d$, we can consider $v:\R^d\to \R$ a periodic function
 and $H:\R\times T^*\R^d\to\R$ periodic in $x$. Then in order that
 $L=\vp_H^{s,t}(dv)$ admits a G.F.Q.I., it is enough to require that the
 flow $\vp_H^{s,\tau}$ is well-defined for $\tau\in [s,t]$. Indeed,
 since $dv$ is compact, $\bigcup_{\tau\in
 [s,t]}\{\tau\}\times\vp_H^{s,\tau}(dv)$ is compact, hence the
 condition of finite propagation speed is satisfied automatically.
 \end{ex}

 \begin{ex} The following hypotheses yield the finite propagation speed property:
 \[
 |\p_y H|\leq
 C_H'(1+|x|),\quad |\p_x H|\leq C_H(1+|y|)\/,
 \]
It is a classical condition for the existence and uniqueness of viscosity solutions in
 $\R^d$, see \cite{GL2}.
 \end{ex}

 \subsection{Generalized generating families and minmax in the Lipschitz setting}
 
 Already if $d=1$\/, $H(t,x,y)=\frac{1}{2}y^2$ and $v(x)=\arctan x$\/, the Lagrangian submanifold $\vp_H^{0,t}(dv)=\left\{\left(x+\frac{t}{1+x^2},\frac{1}{1+x^2}\right):x\in\R\right\}$ is not the graph of a function for $t>0$ large enough\/, and the minimax of its generating family $S_t(x;x_0,y_0)=\arctan x_0+\frac{t}{2}y_0^2+(x-x_0)y_0$ is not a $C^1$ function, though it is locally Lipschitzian (see Proposition~\ref{lemun47} herafter).
 
 Hence, in order to iterate the minmax procedure, one is led to
 defining the minmax when the Cauchy datum is a Lipschitzian
 function. We will use Clarke's generalization of the derivatives of
 $C^1$ functions in the Lipschitz setting \cite{FHb}, see 
 \ref{appA}.

 \begin{prop}
 Under the hypothesis (\ref{equnun}) and with the notation of Corollary~\ref{comp3}, if $v$ is only locally Lipschitzian, the family $S$ given by (\ref{nnn}) generates $L=\vp_H^{s,t}(\p v)$ in the sense that
 \be\label{prop134}
 L=\big\{\big(x,\p_x S(x;\eta)\big)\big|0\in \p_{\eta}S(x;\eta)\big\},
 \ee
 where $\p$ denote Clarke's generalized derivative and $\p v:=\{(x,p),p\in \p v(x)\}$.
 \end{prop}
 
 \begin{proof}
 The equation $0\in \p_{\eta}S(x;\eta)$ means that $y_0\in \p v(x_0)$
 and $y_{i+1}=y_i+\p_{x_{i+1}}\phi_H^{t_i,t_{i+1}}(x_{i+1},y_i)$\/,
 $x_{i}=x_{i+1}+\p_{y_i}\phi_H^{t_i,t_{i+1}}(x_{i+1},y_i)$ for $0\leq
 i<N$\/,where $x:=x_N$ et $\eta=(x_i,y_i)_{0\leq i<N}$.
 \end{proof}
 
 However, this definition of a generating family is not invariant by fiberwise diffeomorphism, even by the following very simple (and useful) one:
 \[
 \big(x;(x_i)_{0\leq i<N},(y_i)_{0\leq i<N}\big)\mapsto\big(x,(x_{i+1}-x_i,y_i)_{0\leq i<N}\big)=:\big(x,(\xi_i,y_i)_{0\leq i<N}\big)\/;
 \]
 indeed, it transforms the family $S$ given by \eref{nnn} into
 \[\fl
 S'\big(x;(\xi_i,y_i)_{0\leq i<N}\big):=v\Big(x-\sum_{0\leq i<N}\xi_i\Big)+ \sum_{0\leq i<N} \phi_H^{t_i,t_{i+1}}\Big(x-\sum_{i<j<N}\xi_j,y_i\Big)+ \sum_{0\leq i<N}\xi_i y_i\,,
 \]
 for which $\p_xS'\big(x;(\xi_i,y_i)_{0\leq i<N}\big)$ is not a point, but the subset
 \[
 \p v\Big(x-\sum_{0\leq i<N}\xi_i\Big)+ \sum_{0\leq i<N} \p_1\phi_H^{t_i,t_{i+1}}\Big(x-\sum_{i<j<N}\xi_j,y_i\Big)\/.
 \]
 As often, this difficulty is overcome by finding the right definition\footnote{But this example exhibits one of the features of  the Clarke derivative: the relation $(y,0)\in\p S'(x,\eta)$ is definitely not equivalent to $y\in\p_x S'(x,\eta)$\/, $0\in\p_{\xi_i} S'(x,\eta)$ and $0\in\p_{y_i} S'(x,\eta)$\/.}:

 \begin{defn}
 \label{nonse}
 A Lipschitz family $S:\R^d\times \R^k\to \R$ is called a \emph{generating
 family} for $L\subset T^*\R^d$ when
 \[
 L=\{(x,y)\in T^*\R^d| \exists \eta\in\R^k: (y,0)\in \p S(x,\eta)\}\/.
 \]
 \end{defn}
 
 \begin{lem}
 This definition of a generating family is invariant by fiberwise
 $C^1$ diffeomorphisms.
 \end{lem}
 
 \begin{proof}
 If $\Phi(x,\eta')=\big(x,\phi(x,\eta')\big)$ is a fiberwise
 diffeomorphism of $\R^d\times \R^k$, and $S':=S\circ\Phi$, then the
 chain rule (see \ref{appA}, Lemma~\ref{B12}) yields
 \[
 \p S'(x,\eta')=\textstyle\left\{\left(y+\zeta
 \frac{\p}{\p x}\phi(x,\eta'),\zeta \frac{\p}{\p\eta'}\phi(x,\eta')\right)\Big|(y,\zeta)\in \p S\big(x,\phi(x,\eta')\big)\right\};
 \]
 as $\eta'\mapsto \phi(x,\eta')$ is a diffeomorphism, it does follow that the two conditions
 \[
 \exists \eta\in\R^k: (y,0)\in \p S(x,\eta)\quad\hbox{and}\quad\exists \eta'\in\R^k: (y,0)\in \p S'(x,\eta')
 \]
 are equivalent.
 \end{proof}
 
 We are now ready to consider G.F.Q.I.'s for the elements of
 \[
 \tilde{\mathcal{L}}:= \{L=\vp(\p v),\quad v\in C^{\Lip}(\R^d),
 \,\vp\in Ham_c(T^*\R^d)\}:
 \]
 
 \begin{prop}
 \label{propun39}
 If $H:[0,T]\times T^*\R^d\to\R$ is $C^2$ and has compact support, then for each $v\in C^{\Lip}(\R^d)$\/, the generating family of $L=\vp_H^{s,t}(\p v)\in\tilde{\mathcal{L}}$ given by \eref{nnn}, namely
 \[
 S(x;\eta)= v(x_0)+ \sum_{0\leq i<N} \phi_H^{t_i,t_{i+1}}(x_{i+1},y_i)+ \sum_{0\leq i<N}(x_{i+1}-x_i)y_i\/,
 \]
 where $x_N:=x$, $\eta:=\big((x_i,y_i)\big)_{0\leq i<N}$\/,  is ``quadratic at infinity'' in the following sense: let
 \[
 Q(\eta):=-x_{N-1}y_{N-1}+\sum_{0\leq i<N-1}(x_{i+1}-x_i)y_i,
 \]
 the Lipschitz constant of each $S(x,\cdot)-Q:(T^*\R^d)^N\to\R$ is bounded, uniformly with respect to $x$ on each compact subset of $\R^d$\/.
 
 Hence, for each compact $K\subset\R^k$\/, if $\theta\in
 C^\infty_c(\R^d,[0,1])$ equals $1$ in a neighbourhood of $0$\/,
 there exists a positive constant $a_K$ such that the function \be
 \label{equndix}\fl S_K(x;\eta)=\psi_K(x;\eta)+Q(\eta)\/,\hbox{ where }
 \psi_K(x;\eta):=\textstyle\theta\left(\frac{\eta}{a_K}\right)\big(S(x,\eta)-Q(\eta)\big)\/,\;x\in
 K. \ee is a G.F.Q.I. of $L_K:=L\cap\pi^{-1}(K)$ in that sense.
 \end{prop}
 
\begin{proof} Denote $\psi(x,\eta)=S(x,\eta)-Q(\eta)$, and $Q(\eta)=\frac{1}{2}\eta^TB\eta$. For a fixed compact subset $K$, let $c=\max_{x\in K}\Lip(\psi(x,\cdot))$, and assume that $|D\theta|\leq 1$. By Lemma \ref{B16}, 
\beaa \p_{\eta}S_K(x,\eta)&=&\p_{\eta}(\theta(\frac{\eta}{a_K})\psi(x,\eta)+Q(\eta))\\&\subset& \frac{1}{a_K}D\theta(\frac{\eta}{a_K})\psi(x,\eta)+ \theta(\frac{\eta}{a_K})\p_{\eta}\psi(x,\eta)+ DQ(\eta),\eeaa
We have 
\[|\psi(x,\eta)|\leq |\psi(x,0)|+|\psi(x,\eta)-\psi(x,0)|\leq b+ c|\eta|\]
where $b:=\max_{x\in K}|\psi(x,0)|$. Hence,
\[\fl |\frac{1}{a_K}D\theta(\frac{\eta}{a_K})\psi(x,\eta)+ \theta(\frac{\eta}{a_K})\p_{\eta}\psi(x,\eta)|\leq \frac{1}{a_K}(b+c|\eta|)+c\leq \frac{1}{2}|B^{-1}|^{-1}|\eta|< |DQ(\eta)|,\]
when $|\eta|\geq b_K$, for some $b_K$ with $a_K,b_K$ large enough. In addition, we can choose $a_K$, $b_K$ such that for $|\eta|\leq b_K$, $\theta(\frac{\eta}{a_K})=1$. Thus $S_K=S$ for $|\eta|\leq b_K$ and there are no critical points of $S$,$S_K$ outside $\{|\eta|\leq b_K\}$, from which $L_K=\{(x,\p_x S_K(x,\eta))|0\in \p_{\eta} S_K(x,\eta)\}$.
\end{proof}
 
\emph{In the sequel, unless otherwise specified, we consider families $S$ of the form (\ref{nnn}) or (\ref{gfqi}) and families $S_K$ of the form (\ref{equndix})}. The advantage is that  $S$  generates $L$ in the simple sense (\ref{prop134}), which helps to express the properties of minmax $R_S(x)$ in a clear way,  similar to the $C^2$ case.
 
 To study the minmax function $R_S$ for such $S$, we use the extension of classical results in critical point theory to locally Lipschitz functions described in \ref{appA}.
 
 \begin{prop}
 \label{propun44}
 The minmax $R_S(x)$ is well-defined and it is a critical value\footnote{\ref{appA}, Definition~\ref{crit}.} of
 the map $\eta\mapsto S(x,\eta)$\/. For each compact subset $K$ of $\R^d$ and each truncation $S_K$ of $S$ of the form \eref{equndix} generating $L_K$\/, we have that $R_S(x)=R_{S_K}(x)$ for $x\in K$\/.
 \end{prop}
 
 \begin{proof}
 By Proposition~\ref{propun39}, $f(\eta):=S(x,\eta)=\psi(\eta)+Q(\eta)$ with $\psi$ Lipschitzian and $Q$ a nondegenerate quadratic form. Hence, $f$ satisfies the P.S. condition (\ref{appA}, Example~\ref{exB7}). If $c=R_S(x)$ were not a critical value,  the flow $\vp_V^t$ of  Theorem~\ref{def1} in \ref{appA} would deform the descending cycles in $f^{c+\ep}$ into descending cycles in $f^{c-\ep}$, hence the contradiction $c=\inf\max_{\sigma} f\leq c-\ep$\/.
 
 To see that $R_S|_K=R_{S_K}$\/, just notice that every descending
 cycle $\sigma$ of $S(x,\cdot)$ or $S_K(x,\cdot)$\/, $x\in K$\/, can
 be deformed into a common descending cycle $\sigma'$ with $\max
 S\big(x,\sigma'(\cdot)\big)=\max S_K\big(x,\sigma'(\cdot)\big)$ by
 using the gradient flow of $Q$\/, suitably truncated.
 \end{proof}
 
 \goodbreak
 \begin{prop}
 \label{lemun47}
 The minmax $R_S(x)$ is a locally Lipschitz function.
 \end{prop}
 
 \begin{proof}
 Let $K\subset\R^d$ be compact. By Proposition~\ref{propun44}, we have that $R_S|_K=R_{S_K}$\/, where $S_K:K\times\R^k\to\R$ writes $S(x,\eta)=\psi_K(x,\eta)+Q(\eta)$ with $Q$ a nondegenerate quadratic form and $\psi_K$ a compactly supported Lipschitz function. Given $x,x'\in K$, for all $\ep>0$, there exists a descending cycle $\bar{\sigma}$ such that $\max_{\eta\in \bar{\sigma}} S_K(x,\eta)\leq R_{S}(x)+ \ep$; if $\max_{\eta\in
 \bar{\sigma}} S_K(x',\eta)$ is reached at $\bar{\eta}$, then
 \beaa
 R_{S}(x')-R_{S}(x)&\leq& S_K(x',\bar{\eta})-S_K(x,\bar{\eta})+
 \ep=\psi_K(x',\bar{\eta})-\psi_K(x,\bar{\eta})+\ep
 \\
 &\leq& \Lip(\psi_K)|x-x'|+\ep \/.
 \eeaa
 If we let $\ep\to 0$ and exchange $x$ and $x'$, we obtain
 \[
 |R_S(x)-R_S(x')|\leq \Lip(\psi_K)|x-x'|,
 \]
 which proves our result.
 \end{proof}
 
 \begin{prop}
 The sets $C(x)=\{\eta\mathop{|} 0\in \p_{\eta} S(x,\eta), S(x,\eta)=R_S(x)\}$ are compact\footnote{\ref{appA}, Example~\ref{exB7}.} and the set-valued map \emph{(``correspondence'')} $x\mapsto C(x)$ is \emph{upper semi-continuous}:  for every convergent sequence $(x_k,\eta_k)\to (x,\eta)$ with $\eta_k\in C(x_k)$, one has $\eta\in C(x)$\/. In other words, the graph $C=\{(x,\eta)\mathop{|}\eta\in C(x)\}$ of the correspondence is closed.
 \end{prop}
 
 \begin{proof}
 Let $(x_k,\eta_k)\to (x,\eta)$ with $\eta_k\in C(x_k)$\/; then, as
 $S$ is $C^1$ in $x$, $\p S=\p_x S\times \p_{\eta}S$. Now $\p
 S:(x,\eta)\mapsto \p_x S\times\p_{\eta}S$ is upper semi-continuous
 (\ref{appA}, Proposition~\ref{lemB11}), the limit
 $\big(\p_x S(x,\eta),0\big)$ of the sequence $\big(\p_x
 S(x_k,\eta_k),0\big)\in\p S(x_k,\eta_k)$ belongs to $\p
 S(x,\eta)$\/, hence $0\in \p_{\eta} S(x,\eta)$\/; as the continuity
 of  $S$ and $R_S$ implies that $S(x_k,\eta_k)\to S(x,\eta)$ and
 $R_S(x_k)\to R_S(x)$\/, this proves $\eta\in C(x)$\/.
 \end{proof}

 \begin{lem}
 \label{simd}
 Given any $\delta>0$, there exists an $\ep>0$ such that
 \[
 R_S(x)=\inf_{\sig\in\Sigma_{\ep}}\max_{\sig\cap C_{\delta}(x)}S(x,\eta)
 \]
 where $\Sigma_{\ep}=\{\sig\mathop{|} \max_{\sig} S(x,\eta)\leq R_S(x)+\ep\}$ and $C_{\delta}(x)=B_{\delta}(C(x))$
 denotes the $\delta$-neighborhood of the critical set $C(x)$.
 \end{lem}
 
 \begin{proof}
 This is a direct consequence of the deformation lemma (\ref{appA}, Theorem~\ref{def2}) for $S_x:=S(x,\cdot)$: for $\delta>0$, and $c=R_S(x)$, there exist $\ep>0$ and $V$ such that $\vp_V^1(S_x^{c+\ep}\setminus C_{\delta}(x))\subset S_x^{c-\ep}$. In particular, we remark that for  $\sigma\in \Sigma_{\ep}$, the intersection $\sigma\cap C_{\delta}(x)$ is nonempty, otherwise, the flow $\vp_V^1$ may take $\sigma$ to a descending cycle $\sigma'=\vp_V^1(\sigma)$ such that $\max_{\eta\in \sigma'}S_x(\eta)\leq R_S(x)-\ep$, contradiction with the definition of minmax.
 \end{proof}
 
 \begin{rem} When $S$ is $C^2$, the $S_x$'s are generically Morse functions: indeed, $S_x$ is Morse if and only if $x$ is a regular value of the projection $\pi: L\to M,\,(x,p)\mapsto x$, whose regular values, by Sard's theorem and the compactness of $\hbox{Crit}(S_x)$, form an open set of full measure. In this case, $S_x^c$ is indeed a deformation retract of $S_x^{c+\ep}$ for $\ep>0$ small enough, hence $\inf\max$ deserves its name ``minmax'', that is, there exists a descending cycle $\sigma$ such that, $R_S(x)=\max_{\sigma} S(x,\eta)=\max_{\sigma\cap C(x)} S(x,\eta)$.
 \end{rem}
 
 \begin{prop}\label{par} The generalized derivative of $R_S$
 satisfies
 \be
 \label{pss}\p R_S(x)\subset \conv\{\p_x S(x,\eta)\mathop{|}\eta\in C(x)\}.
 \ee
 \end{prop}
 
 \begin{proof}
 First, we claim that, if $R_S$ is differentiable at $\bx$, then
 \be
 \label{derivable}
 dR_S(\bx)\subset \conv\{\p_x S(\bx,\eta)\mathop{|}\eta\in C(\bx)\}
 \ee
 Take $\delta$ and $\ep$ for $\bx$ as in Lemma \ref{simd}. Consider $K=\overline{B_1(\bx)}$, and $S_K$ obtained in Proposition \ref{propun39}, one can choose a  $\varrho\in (0,1)$ such that for
 $x\in B_{\varrho}(\bx)$,
 \[
 |S_K(x,\cdot)-S_K(\bx,\cdot)|_{C^0}\leq \ep/4\/.
 \]

 Now let $y\in \R^d$ and $\lambda<0$ be small so that $x_{\lambda}:=\bx+ \lambda y\in B_{\varrho}(\bx)$ and $\lambda^2<\ep/4$. Then by Lemma \ref{simd}, for each $x_{\lambda}$, there is a descending cycle $\sigma_{\lambda}$ such that
 \[
 \max_{\sigma_{\lambda}}S(x_{\lambda},\eta)\leq R_S(x_{\lambda})+\lambda^2,
 \]
 then,
 \[
 \max_{\sigma_{\lambda}} S(\bx,\eta)\leq \max_{\sigma_{\lambda}}S(x_{\lambda},\eta)+ \frac{\ep}{4}\leq
 R_S(x_{\lambda})+\frac{\ep}{2}\leq R_S(\bx)+\frac{3\ep}{4}
 \]
 and
 \[
 R_S(\bx)\leq \max_{\sigma_{\lambda}\cap C_{\delta}(\bx)}S(\bx,\eta)=S(\bx,\eta_{\lambda}),\quad \mbox{for\, some\, }\eta_{\lambda}\in \sigma_{\lambda}\cap C_{\delta}(\bx)\/.
 \]
 Hence we have
 \bea
 \label{eqs}
 \lambda^{-1}[R_S(x_{\lambda})-R_S(\bx)]&\leq& \lambda^{-1} [S(x_{\lambda}, \eta_{\lambda}) - S(\bx,\eta_{\lambda})] - \lambda
 \\
 &=& \l \p_x S(x_{\lambda}',\eta_{\lambda}),y\r-\lambda,
 \eea
 where the last equality is given by the mean value theorem for some $x_{\lambda}'$ in the line segment between $\bx$ and $x_{\lambda}$.
 
 Take the $\limsup$ of both sides in the above inequality and let $\delta\to 0$, we get
 \[
 \l d R_S(\bx),y\r   \leq \max_{\eta \in C(\bx)} \l \p_x
 S(\bx,\eta),y \r,\quad \forall y\in \R^d
 \]
 
 Note that this implies that $d R_S(\bx)$ belongs to the
 sub-derivative of the convex function $f(y):=\max_{\eta \in C(\bx)}
 \l \p_x S(\bx,\eta),y \r$ at $v=0$,\footnote{For a
 convex function $f$, the sub-derivative at a point $x$ is the set of
 $\xi$ such that $f(y)-f(x)\geq \l \xi,y-x\r$,$\forall y$} for which
 one can easily calculate
  \[ \p f(0) = \conv \{\p_x S(\bx,\eta): \eta
 \in C(\bx)\}.\] Thus we get (\ref{derivable}). In general, \beaa \p
 R_S(x)&=&\conv\{\lim_{x'\to x} dR_S(x')\}\subset\conv\{\conv
 \lim_{x'\to x} \{\p_x S(x',\eta'),\eta'\in
 C(x')\}\}\\&\subset&\conv\{\p_x S(x,\eta), \eta\in C(x)\}\eeaa by
 the upper-semi continuity of $x\mapsto C(x)$ and the continuity of
 $\p_x S$\/.
 \end{proof}
 
 The formula (\ref{pss}) gives us somehow a {\it generalized graph selector}. While for a classical graph selector, we require that for almost every $x$,
 \[
 dR_S(x)=\p_x S(x,\eta),\quad \mbox{for\,some\,}\eta\in C(x)
 \]
 from which $(x,dR_S(x))\in L$. Following Chaperon, Sikorav \cite{loi,PGS}, this occurs if $S$ is a G.F.Q.I.  of $L=\vp(dv)\in \mathcal{L}$ for $v\in C^2$. Indeed, in this case,  $S_x:=S(x,\cdot)$ is an excellent Morse function for almost every $x$, in which cases $C(x)$ consists of a single point, hence $\p
 R_S(x)=\p_x S(x,\eta)$ for a unique $\eta$, proving that $R_S$ is a true graph selector for $L$. 
\vspace{6pt}

There are still some mysteries for the generalized graph selector. One may ask whether the minmax $R_S$ is also a true graph selector for $L\in \tilde{\mathcal{L}}$. Is it true that, when $R_S$ is differentiable at $x$, one has $\big(x, dR_S(x)\big)\in L$\/, where $L\in \mathcal{L}$ or even $\tilde{\mathcal{L}}$?

 \section{Viscosity solutions and minmax solutions}
 We  look at the solutions of the (H-J) Cauchy problem, assuming that $H\in C^2([0,T]\times T^*\R^d)$ and $v\in C^{\Lip}(\R^d)$ satisfy the condition of finite propagation
 speed. Unless otherwise specified, we assume that $H$ has compact support (as a function on $[0,T]\times T^*\mathbb{T}^d$ when $H$ and $v$ are periodic).
\subsection{Geometric solution and its minmax selector}\label{charat}
 Following the classical geometric method for first order partial differential equations, the Hamilton-Jacobi equation is considered to be a hypersurface in the cotangent bundle $T^*(\R\times \R^d)$. 
 See \cite{Arnold1,Arnold2}.
 
 More precisely, let
 \[
 \mathcal{H}(t,x,e,p)=:e + H(t,x,p),\quad (t,x,e,p)\in T^*(\R \times \R^d)
 \]
 and  at the moment suppose that the initial function $v$ is $C^2$.
 
 \begin{defn}
  Let $\vp_{\mathcal{H}}^s$ denote the Hamiltonian flow of $\mathcal{H}$, which preserves the levels of $\mathcal{H}$, and let
 \[
 \Gamma_{v}=\left\{\left(0,x,-H\big(0,x,dv(x)\big),dv(x)\right)\right\}\,;
 \]
 then, the  \emph{geometric solution} of the Cauchy problem (H-J) is
 \[
L_{\mathcal{H},v}:=\bigcup_{s\in [0,T]} \vp_{\mathcal{H}}^s(\Gamma_{v})\/.
 \] 
It is a Lagrangian submanifold  containing the initial isotropic submanifold $\Gamma_{v}$
 and  contained in the hypersurface
 \[
 \mathcal{H}^{-1}(0)=\{(t,x,e,p)|e + H(t,x,p)=0\}\subset T^*(\R\times \R^d).
 \]
 As every Lagrangian submanifold $L$ of $T^*(\R \times \R^d)$ contained in $\mathcal{H}^{-1}(0)$ is locally invariant by $\vp_{\mathcal{H}}^s$\/, this geometric solution is in some sense maximal.
 \end{defn}

 Writing  $T^*(\R\times \R^d)$ as $T^*\R\times T^*\R^d$, we have $X_{\mathcal{H}}=(1,-\p_t H, X_H)$, and
 \[
 L_{\mathcal{H},v}=\bigcup_{t\in [0,T]}\left\{\left(t,-H\big(t,\varphi_H^{t}(dv)\big),\varphi_H^{t}(dv)\right)\right\}
 \]
 where $\vp_H^t:=\varphi_H^{0,t}$ is the Hamiltonian isotopy generated by  $H$.
 
 \begin{lem}
 Formula (\ref{gfqi}) defines a G.F.Q.I. of $L_{\mathcal{H},v}$.
 \end{lem}
 
 \begin{proof}
 For simplicity, we may assume that $T\in (0,\delta_H)$, hence that
 \[\fl
 S:[0,T]\times \R^d\times \R^{2d}\to \R,\quad S(t,x,x_0,y_0)=v(x_0)+ xy_0+ \phi_H^t(x,y_0)-x_0y_0\/.
 \]
 
 Let $(x_0,y_0)\in \Sigma_S$, then
 \[\fl
 (\p_t S(t,x,x_0,y_0),\p_x S(t,x,x_0,y_0))=(\p_t
 \phi_H^t(x,y_0),\p_x\phi_H^t(x,y_0))=(-H(t,x,y(t)),y(t))\/,
 \]
 where $(x,y(t))=\vp_H^t(x_0,y_0)$ with $y_0=dv(x_0)$.
 
 Hence
 \[
 \{(t,x,\p_t S(t,x,x_0,y_0),\p_x S(t,x,x_0,y_0))|(x_0,y_0)\in \Sigma_S\}=L_{\mathcal{H},v}
 \]
 \end{proof}
 The {\it wave fronts} of the geometric solution $L_{\mathcal{H},v}$ are then given by
 \[\mathcal{F}^t:=\{(x,S(t,x;\eta))|\frac{\p S}{\p \eta} (t,x;\eta)=0\}\subset J^0(\R^d).\]

 If there exists a $C^1$ function $u:[0,T]\times \R^d\to \R$ such that
 \[
 L=L_{\mathcal{H},v}=\{(t,x,\p_t u(t,x),\p_x u(t,x))\subset T^*([0,T]\times \R^d)\}
 \]
 we say that $L$ is a 1-graph in $T^*([0,T]\times \R^d)$. In this case, $u$ is a global solution of the Cauchy problem of (H-J) equation. In general, $L$ may be the graph of the derivatives of a multi-valued function.

 An equivalent but more economic way to describe the geometric
 solution is to identify (as in the introduction) each  $\vp_{\mathcal{H}}^s(\Gamma_v)$ with
 $\{s\}\times \vp_H^s(dv)$ by the inverse of the map $(t,x,p)\mapsto (t,x,-H(t,x,p),p)$.
 In this way, we also call the union
 \[
 L_{H,v}:=\bigcup_{t\in [0,T]}\{t\}\times \vp_H^t(dv)\subset \R\times
 T^*\R^d
 \]
the {\it geometric solution}.
 
 If we look at the projection of the characteristics, that is, the image of the graph of the solutions $\big\{\big(t,\vp_H^t(x_0,p_0)\big)\big\}_{t\in [0,T]}$, $(x_0,p_0)\in T^*\R^d$, of Hamilton's equations under
 the projection
 \[
 \pi: [0,T]\times T^*\R^d\to \R\times \R^d,\quad (t,x,p)\mapsto (t,x)\/.
 \]
 then $L$ is not a 1-graph when the corresponding  characteristics intersect under the projection. Without ambiguity, we will simply say that the characteristics intersect.\\
 
 \noindent For the initial functions $v\in C^{\Lip}(\R^d)$, set
\[ 
L_{H,v}=\bigcup_{t\in [0,T]}\{t\}\times \vp_H^t(\p
 v):=\bigcup_{t\in [0,T]}\{t\}\times \{\vp_H^t(x,p):p\in \p v(x)\}\/,
 \]
 where $\p$ is Clarke's generalized derivative. We call them {\it generalized geometric solutions}. They are also generated by the G.F.Q.I.'s given by formula \eref{gfqi}. 
 
 \begin{defn} 
 For any time $0\leq s<t\leq T$, we define the minimax operator\footnote{The inclusion $R_H^{s,\tau}\big(C^{\Lip}(\R^d)\big)\subset C^{\Lip}(\R^d)$ is proven in Proposition~\ref{prof} hereafter.}
 \[
 R_H^{s,\tau}:C^{\Lip}(\R^d)\to C^{\Lip}(\R^d),\quad \tau\in [s,t]
 \]
  for the (H-J) equation as
 \[
 R_H^{s,\tau} v(x)=\inf\max_{\eta} S(\tau,x,\eta)
 \]
 where $S:[s,t]\times \R^d\times \R^k\to \R$ is given by (\ref{gfqi}).
 \end{defn}
 
 For completeness, without referring to the uniqueness theorem for G.F.Q.I.'s, we give a proof that the minmax is well-defined independently of the subdivisions.

 \begin{lem}
 \label{ides}
 The minmax $R_S(x)=\inf\max S(x,\eta)$ given by \eref{nnn} or \eref{gfqi} is independent of the subdivision of
 time in the construction of $S$.
 \end{lem}
 
 \begin{proof}
 First assume  $t-s<\delta_H$; given $\tau\in (s,t)$, consider the family of subdivisions $\zeta_{\mu}:=\{s\leq s+\mu(\tau-s)<t\}$; then, 
 \[\fl
 S_{\mu}(x;x_0,y_0,x_1,y_1)=v(x_0)+\phi_H^{s,s+\mu(\tau-s)}(x_1,y_0)+(x_1-x_0)y_0+\phi_H^{s+\mu(\tau-s),t}(x,y_1)+(x_2-x_1)y_1\,,
 \]
 where $x_2:=x$, is the generating family defined by (\ref{nnn}) and associated to $\zeta_{\mu}$, $\mu\in (0,1]$. The function $S_{\mu}$ is continuous in $\mu$ and the minmax $R_{S_{\mu}}(x)$ is a critical value of the map $\eta\mapsto S_{\mu}(x;\eta)$ with $\eta:=(x_0,y_0,x_1,y_1)$. By \eref{equncinq}, the set of all such critical values is independent of $\mu$; as it has measure zero by Sard's Theorem, $R_{S_{\mu}}$ is constant for $\mu\in [0,1]$. In particular, letting $x_1':=x_1-x_0$ and $y_0'=y_0-y_1$, we get
 \[
 S_0=S_0(x;(x_0,y_1,x_1',y_0'))=v(x_0)+\phi_H^{s,t}(x_2,y_1)+(x_2-x_0)y_1+x_1'y_0'.
 \]
 It is obtained by adding the quadratic form $x_1'y_0'$ to 
 \[
 S(x;x_0,y_1)=v(x_0)+\phi_H^{s,t}(x_2,y_1)+(x_2-x_0)y_1,
 \] 
 which is the generating family related to $\zeta_0$.
 We conclude that \[R_S(x)=R_{S_0}(x)=R_{S_1}(x).\]
 
 In general, given any two subdivisions $\zeta'$, $\zeta''$ of $[s,t]$ with\footnote[1]{For a subdivision $\zeta=\{t_0<\cdots<t_n\}$, we let $|\zeta|:=\max_{i}|t_{i+1}-t_i|$.} $|\zeta'|,|\zeta''|<\delta_H$\/,  denote by $\zeta=\zeta'\cup\zeta''=\{s=t_0<\cdots<t_n=t\}$ the subdivision obtained by collecting the points in $\zeta'$ and $\zeta''$. If $t_j$ is not contained in $\zeta'$, we consider the family of subdivisions 
 \[
 \zeta_{\mu}(j)=\{t_0<t_{j-1}\leq t_{j-1}+\mu(t_j-t_{j-1})<t_{j+1}<\cdots t_n\},\quad \mu\in[0,1]
 \]
 The same argument as before shows that the minmax relative to $\zeta_0(j)$ and $\zeta_1(j)$ are the same. Continuing this procedure, we get that the minmax relative to $\zeta'$ and $\zeta$ are the same, and the same holds for $\zeta''$ and $\zeta$. Therefore the minmax with respect to $\zeta'$ and $\zeta''$ are the same.
 \end{proof}

 \begin{lem}
 If $v\in C^2\cap C^{\Lip}(\R^d)$, then $R_H^{0,t}v(x)$ verifies the (H-J) equation
 almost everywhere.
 \end{lem}
 
 \begin{proof}
 This is a direct consequence of the fact that $S$ is a G.F.Q.I. of $L_{\mathcal{H},v}$
 and the minmax is a graph selector in this case.
 \end{proof}
 
 In general, for a Lipschitzian initial function, we do not know whether the minmax verifies
 the equation almost everywhere or not. But in view of the estimation of generalized derivatives in Proposition \ref{par}, and that $(x,R_H^{0,t}v(x))\in \mathcal{F}^t$, we call $u(t,x)=R_H^{0,t}v(x)$ the {\it minmax solution} of the (H-J) Cauchy problem.
 
 \begin{lem}
 \label{cara}
 If $v$ is $C^2$ with bounded second derivative, then there exists an $\ep>0$ such that for $t\in [0,\ep)$, the minmax $R_H^{0,t}v(x)$ is $C^2$.
 \end{lem}
 
 \begin{proof}
 We will show that, there exists an $\ep>0$, such that for $t\in (0,\ep)$,
 the characteristics beginning from the graph $dv$ do not intersect.  More precisely, the map $f_t:x_0\mapsto X_0^t\big(x_0,dv(x_0)\big)$ is a diffeomorphism. Indeed, for $t$ small enough,
 \[
 \Lip(f_t-Id)\leq \Lip(\alpha_0^t-Id)(1+\Lip (dv))<1
 \]
where $\alpha_0^t$ and $c_H$ are defined in Lemma \ref{1.1}. This in turn means that the projection map $L=\vp_H^t(dv)\to \R^d$, $(x,p)\mapsto x$ is a diffeomorphism,
 hence $L=\{x,dR_H^{0,t}v(x)\}$, from which we obtain that $R_H^{0,t}v(x)$ is $C^2$.
 \end{proof}
 
 \subsection{Viscosity solutions}

 \begin{defn}
 A function $u\in C^0\big((0,T)\times \R^d \big)$ is called a \emph{viscosity subsolution} (resp. \emph{supersolution}) of
 \[
 \p_t u + H(t,x, \p_x u)=0
 \]
 when it has the following property: for every $\psi\in C^1\big((0,T)\times \R^d \big)$ and every
 point $(t,x)$ at which $u-\psi$ attains a local maximum (resp. minimum), one has
 \[
 \p_t \psi + H(t,x,\p_x \psi)\leq 0,\quad (\hbox{resp}. \geq 0)\/.
 \]
 The function $u$ is  a \emph{viscosity solution} if it is both a viscosity subsolution and supersolution.
 \end{defn}
 
 \goodbreak

We remark that one can replace $C^1$ test functions $\psi$ by $C^{\infty}$ test functions in the definition. Obviously, a classical $C^1$ solution is a viscosity solution.

 \begin{thm}[\cite{GL2}]
 If $v\in C^{\Lip}(\R^d))$ and $H\in C^2_c([0,T]\times T^*\R^d)$, then there exists a unique
 viscosity solution  of the Cauchy problem of the Hamilton-Jacobi equation. Moreover, this solution
 is globally Lipschitz.
 \end{thm}
 
 A notable feature of the viscosity solution, is the {\it Markov property}: if
$J_s^t:C^{\Lip}(\R^d)\to C^{\Lip}(\R^d)$
 denotes the viscosity solution operator (for a fixed Hamiltonian) which to $v$ associates the time $t$ of the solution equal to $v$ at time $s$, then the ``two-parameter groupoid'' property
 $J_\tau^t=J_s^t\circ J_\tau^s$
 is satisfied. This follows easily from  uniqueness.
 
 \vspace{4pt}
 
 The following Proposition summarizes a well-studied case when the Hamiltonian is convex in $p$, one can refer to \cite{T.J,FG, these}.
 
 \begin{prop} \label{minsemi}Assume  $H\in C^2([0,T]\times T^*\R^d)$  strictly convex in $p$, equal to $|p|^2$ off a compact set, and $v\in C^{\Lip}(\R^d)$. Then the minmax solution is reduced to a min, and it possesses the ``semi-group'' property with respect to time, that is
 \[R_0^tv(x)=R_s^t\circ R_0^sv(x),\quad 0\leq s\leq t\]
 \end{prop}
 
 \begin{proof} 
 Our hypotheses imply that there exists a constant $\epsilon_H>0$ such that, for $0<t-s<\epsilon_H$, the Hamiltonian diffeomorphism $\vp_s^t$ of $H$ has a ``classical'' generating function $\psi_s^t(X,x)$ in the sense that  $\big((x,y),(X,Y)\big)\in \hbox{Graph}(\vp)$ if and only if
 \[
 \left\{
                     \begin{array}{ll}
                       Y=\p_X \psi(X,x) \\
                       y=-\p_x \psi(X,x)\/.
                     \end{array}
                   \right.
                   \]
Therefore, for any $0\leq s<t\leq T$, the subset $L=\vp_s^t(\p v)$ has the generating family
\begin{equation}
\label{gfmin}
F_s^t(x;(x_i)_{0\leq i\leq j})= v(x_0)+ \Psi_s^t(x,(x_i)):=v(x_0)+ \sum_{0\leq i\leq j} \psi_{\tau_i}^{\tau_{i+1}}(x_{i+1},x_i),
\end{equation}
where $x_{j+1}:=x$ and  $\{s=\tau_0<\tau_1<\dots<\tau_{j+1}=t\}$ is a subdivision of $[s,t]$ such that $|\tau_i-\tau_{i+1}|<\ep_H$, $0\leq i\leq j$. Up to diffeomorphism, $F_s^t$ is quadratic of index $0$ at infinity, since $H$ equals $|p|^2$ off a compact subset. Thus the minmax is reduced to a min:
\[R_s^tv(x)=\min_{(x_i)} F_s^t(x,(x_i)).\] 
Note that $R_s^t$ is independent of the choice of the subdivision (same argument as for Lemma \ref{ides}), hence
 \beaa
 R_s^t\circ R_0^{s}v(x)&=& \min_{(x_i)}
 \left(R_0^s v(x_0)+ \Psi_s^{t} (x,(x_i))\right)\\
 &=& \min_{(x_i)}\Big(\min_{(x_j')}(v(x_0')+
 \Psi_0^s(x_0,(x_j')))+ \Psi_s^{t} (x,(x_i))\Big)\\
 &=&\min_{(x_i),(x_j')} \left( v(x_0')+
 \Psi_0^t(x,(x_i),(x_j'))\right)\\
 &=&R_0^t v(x)\/.
 \eeaa
 \end{proof}
 
 \begin{rem} 
For $S$ and $F$ defined by \eref{gfqi} and \eref{gfmin} respectively, we have 
\[
\inf\max S(x,(x_i,y_i))=\min F(x,(x_i)).
\]
Indeed, when $v$ is $C^2$,  we can conclude by the uniqueness theorem of G.F.Q.I.'s since  $S$ and $F$ generates the same Lagrangian submanifold $L=\vp(dv)$\/; in the general Lipschitz case, we can apply the continuity dependence of the minmax selector on the generating family (ref. Lemma \ref{coi}).
\end{rem}
 
 \begin{rem} 
 The hypothesis that $H=|p|^2$ at infinity can be generalized to the case where the condition of finite propagation speed is satisfied. The min solution operator is a finite dimensional ``discretization'' of the Lax-Oleinik semi-group in weak KAM theory, defined by
 \[
 T_s^tv(x)=\inf_{\gamma(t)=x}\textstyle\big\{v(\gamma(s))+ \int_s^t
 L(t,\gamma(t),\dot{\gamma}(t))dt\big\}
 \]
 where $L$ is the Legendre tranform of $H$ with respect to the $p$ variable, and the $\inf$ is taken over all absolutely continuous paths $\gamma:[s,t]\to \R^d$.
 \end{rem}
 
 \begin{thm}[\cite{T.J}]
 \label{Jouk}
 The min solution $R_0^tv(x)$ is the viscosity solution of the Cauchy problem \emph{(H-J)}.
 \end{thm}

 \subsection{Iterated minmax and viscosity solution}
 
 In contrast to the case of convex Hamiltonians, where the minmax is reduced to a min and provides the viscosity solution, for general non-convex Hamiltonians, the minmax and the viscosity solution may differ: see \cite{Vi2,AC,Cardin} for counterexamples, and also \cite{Alain} for a very nice geometric illustration of the fact that the viscosity solution is not necessarily contained in the geometric solution.
 
Particularly, in \cite{Vi2}, the author pointed out without proof that the minmax does not provide a semi-group as a consequence of not being viscosity. We will make  this point clear by showing that the semi-group property is a sufficient condition for the minmax to be viscosity.
 
 \begin{prop}
 \label{smgp} Given $v$, the minmax $R_H^{0,t} v(x)$ is the viscosity solution of the Cauchy problem \emph{(H-J)} if it has the semi-group property with respect to time, that is,
 \[
 R_H^{0,t}v(x)= R_H^{s,t}\circ R_H^{0,s}v(x),\quad 0\leq s<t\leq T\/.
 \]
 \end{prop}
 
 \begin{proof}
 Suppose $R_0^tv(x):=R_H^{0,t}v(x)$ possesses the semi-group property, we first show that $R_0^t v(x)$ is a viscosity subsolution. For any $(t,x)$, let $\psi$ be a $C^2$ function such that $\psi(s,y)=:\psi_s(y)\geq R_0^sv(y)$, with equality at $(t,x)$. It is enough to consider $\psi$ in a neighborhood of $(t,x)$, where it has bounded second derivative. Then
 \be
 \label{subsol}
 \psi_t(x)=R_{\tau}^t\circ R_0^{\tau}v(x)\leq R_{\tau}^t \psi_{\tau}(x)\/.
 \ee
 By Lemma \ref{cara}, for $t-\tau>0$ small enough,
 the characteristics originating from $d\psi_{\tau}$ do not intersect:
 let $(x_t,y_t)=\vp_{\tau}^t(x_{\tau},\p_x\psi_{\tau}(x_{\tau}))$, where $\vp$ denotes the Hamiltonian
 flow of $H$, then the map
 $p:(x_{\tau},\p_x\psi_{\tau}(x_{\tau}))\mapsto x_t$ is a diffeomorphism.
 Therefore $R_{\tau}^t\psi_{\tau}(x)$ is a classical $C^2$ solution
 of the (H-J) equation.  Hence
 \be \label{hjc} R_{\tau}^t\psi_{\tau}(x)=\psi_{\tau}(x)-\int_{\tau}^t H(s,x,\p_x
 R_{\tau}^s\psi_{\tau}(x))ds\ee Moreover, since $(x,\p_xR_{\tau}^t\psi_{\tau}(x))
 =\vp_{\tau}^t\circ p^{-1}(x)$,
 we get that $\p_xR_{\tau}^t\psi_{\tau}(x)$ is continuous in $\tau$.
 
  Substracting (\ref{hjc}) into (\ref{subsol}), moving
 $\psi_t(x)$ to the right-hand side,  dividing both side by $t-\tau$ and lettting
 $\tau\to t$, we get
 \[0\leq -\p_t\psi_t(x)-H(t,x,\p_x \psi_t(x))\]
 from which we get a subsolution by definition. Similarly, we can
 prove that $R_0^tv(x)$ is a viscosity supersolution.
 \end{proof}
 
 As a direct consequence, we get Theorem \ref{Jouk} since the min solutions form a semi-group (Proposition \ref{minsemi}).
 
 \vspace{6pt}
We remark that Proposition \ref{smgp} does not essentially depend on
the variational formulation of the minmax. Indeed,  an operator is the viscosity solution operator if it verifies the semi-group property, the monotonicity, and it is a generator of the (H-J) equation, that is it generates a regular solution at least within small time when the initial data is regular enough. See for example  \cite{PB} Proposition~20 and \cite{fleming} Theorem 5.1.


%
 To compensate the fact that the minmax is not a semi-group, an idea
 due to M.~Chaperon is to replace the ``minmax''
  by some ``iterated minmax''.
 Roughly speaking, an iterated minmax is obtained by dividing  a given time interval into small pieces
 and taking the minmax step by step. This is a priori a discrete
 semi-group with respect to the points of the subdivision. We are
 going to show that, as the steps of the subdivision go to zero,  the iterated minmax  converges to a genuine semi-group, and therefore to the viscosity
 solution.
 \vspace{6pt}

 In the following, we denote the
 Lipschitz constant of a global Lipschitz function $f$ by $\|\p f\|$ and
 $|\cdot|_K$ denotes the maximum norm on a compact set $K$.
 
 \begin{prop}
 \label{prof}
 Assuming  $H\in C_c^2([0,T]\times T^*\R^d)$ and
 $v\in C^{\Lip}(\R^d)$, we have the
 following estimates:

 $1) $$R_H^{s,t}$ defines an operator from $C^{\Lip}(\R^d)$ to
 $C^{\Lip}(\R^d)$, and
 \[\|\p( R_H^{s,t}v)\|\leq \|\p v\|+ \|\p_x H\|\,|t-s|\]

 $2)$ For any $0\leq s<t_i\leq T$, $i=1,2$,
 \[|R_H^{s,t_1}v(x) -  R_H^{s,t_2}v(x)|\leq
 |t_1-t_2|\max_{t\in [t_1,t_2]} |H(t,x,\cdot)|_Y\] where $Y=\{y:
 |y|\leq \|\p v\|+ \|\p_x
 H\|\max_i |t_i-s|\}$.\\

 $3)$ Let $H^0$ and $H^1$ be two Hamiltonians, then
 \[|R_{H^0}^{s,t}v-R_{H^1}^{s,t}v|_{C^0}\leq |t-s|\max_{\tau
 \in[s,t], y\in Y'}|(H^0-H^1)(\tau,\cdot,y)|_{C^0}\] where
 $Y'=\{y:|y|\leq \|\p v\|+ \max_i\|\p_x H^i\||t-s|\}$.\\

 $4)$ If $v^0, v^1\in C^{\Lip}(\R^d)$ and $K$ is a compact set in
 $\R^d$, then there exists a bounded subset $\tilde{K}\subset \R^d$
 which depends on $K\times[0,T]$ and the constants $\|\p v^i\|$, such
 that \be\label{in}|R_H^{s,t} v^0 -R_H^{s,t} v^1|_{K}\leq
 |v^0-v^1|_{\tilde{K}},\quad 0\leq s<t\leq T.\ee
 \end{prop}

 \begin{proof}
 The proof is based on Proposition \ref{par} with some variation on
 the original variable $x$, which  can be either $t\in [0,T]$,
 $x\in \R^d$ or some parameter $\lambda$ for the generating family
 constructed as below.
 
 For simplicity, we may first assume that $|t-s|<\delta_H$ so that
 \[S^{s,t}(x,x_0,y_0)= v(\x)+ \phi_H^{s,t}(x,\y)+ x\y-\x\y\]
 Let $(x(\tau),y(\tau))$ denote the Hamiltonian flow, and $C(x)$ be the critical
 set defined in Proposition \ref{par}.
 
 1) For $(x_0,y_0)\in C(x)$, we have \[\p_x S^{s,t}(x,x_0,y_0)= \p_x \phi_H^{s,t}(x,y_0)+ y_0= y(t)\] where
 \[y(t)= y_0 - \int_s^t \p_x H(\tau,x(\tau),y(\tau))d\tau, \quad y_0\in \p v(x_0)\]
 Hence by (\ref{pss}),
 \[\p R_H^{s,t}v(x)\subset \hbox{co}\{y(t),\,y_0\in \p v(x_0)\}\]
 thus
 \[\|\p(R_H^{s,t}v)\|\leq  \|\p v\|+ \|\p_xH\||t-s|.\]
 
 2) For $(x_0,y_0)\in C(x)$, by Lemma \ref{dergf}, we have
 \[\p_t S^{s,t}(x,x_0,\y)= \p_t \phi_H^{s,t}(x,\y)= -H(t,x,y(t)).\]
 Hence
 \[
 \p_t R_H^{s,t}v(x)\subset \hbox{co}\{-H(t,x,y(t)),\,y_0\in \p v(x_0)\}
 \]
 and therefore
 \[|R_H^{s,t_1}v(x) -  R_H^{s,t_2}v(x)|\leq
 |t_1-t_2|\max_{t\in[t_1,t_2], y\in Y}|H(t,x,y)|\]
 where $Y=\{y:\,|y|\leq \|\p v\|+ \|\p_x H\|\max_i|t_i-s|\}$.
 \vspace{6pt}
 
 3) Let $H^{\lambda}= (1-\lambda) H^0 +\lambda H^1$, $\lambda\in [0,1]$, and
 let $S_{\lambda}^{s,t}$ be the corresponding generating families. Fix $\lambda$, for
 $(x_0,y_0)$ in the critical set $C^{\lambda}(x)$ corresponding to $H^{\lambda}$,
 
 \[\p_{\lambda}
 S^{s,t}_{\lambda}(x,\x,\y)= \p_{\lambda}
 \phi_{H^{\lambda}}^{s,t}(x,\y)= \int_s^t
 (H^0-H^1)(\tau,x^{\lambda}(\tau),y^{\lambda}(\tau))d\tau.\]
where the proof of the second equality is similar to that of Lemma \ref{dergf}. Hence,
 \[
 \p_{\lambda} R_{H^{\lambda}}^{s,t}v(x)\subset \hbox{co}\{\int_s^t(H^0-H^1)(\tau,x^{\lambda}(\tau),y^{\lambda}(\tau))d\tau,
 \,y_0\in \p v(x_0)\}
 \]
 and therefore
  \beaa
 |R_{H^0}^{s,t}v(x)-R_{H^1}^{s,t}v(x)|&\leq& \int_0^1 \int_s^t
 |H^0-H^1|(\tau,x^{\lambda}(\tau),y^{\lambda}(\tau))d{\tau}d{\lambda}
 \\&\leq& |t-s|\max_{\tau \in[s,t], y\in
 Y'}|(H^0-H^1)(\tau,\cdot,y)|_{C^0}\eeaa where
 $Y'=\{|y|:\,|y|\leq \|\p v\|+ \max_i \|\p_x H^i\||t-s|\}$.\\

 4) Let $v^{\lambda}=(1-\lambda) v^0+ \lambda v^1$, $\lambda\in [0,1]$ and
 $S_{\lambda}^{s,t}$ denotes the corresponding generating families,
 then $\p_{\lambda} S_{\lambda}^{s,t}(x,x_0,y_0)= v^1(x_0)-v^0(x_0)$,
 \[\p_{\lambda}R_H^{s,t} v^{\lambda}(x)\subset
 \conv\{v^1(\x)-v^0(\x):(\x,\y)\in C^{\lambda}(x)\}\] with  $C^{\lambda}(x)\subset\{(x_0,y_0): |\x|\leq |x|+
 T\|\p_y (H|_{\{y\in Y\}})\|\}$, $Y:=\{y:|y|\leq \|\p v\|+
 T\|\p_x H\|\}$. If we take $\tilde{K}=\{\x:|\x|\leq |x|_K+ T\|\p_y
 (H|_{\{y\in Y\}})\|\}$,  we obtain
 \[|R_H^{s,t}v^0-R_H^{s,t}v^1|_K\leq |v^0-v^1|_{\tilde{K}}\]
 
 In general, the above results follow from the fact that the critical
 set $C(x)$ defines the Hamiltonian flow $(x(\tau),y(\tau))_{s\leq
 \tau\leq t}$ for any $0\leq s<t\leq T$.
 \end{proof}
 
 \begin{rem} 
 The estimates in the proposition, more subtle than needed, precisely reveal
 that finite propagation speed is enough to define the minmax function.
 \end{rem}
 \vspace{6pt}
 
 Now given any compact subset $K\subset \R^d$, we consider $(t,x)\in [0,T]\times K$.
 
 Given a subdivision $\zeta=\{0=t_0<t_1<\dots<t_n=T\}$  of $[0,T]$,
 to each $s\in [0,T]$, we associate a number $[s]_{\zeta}$ as:
 $$
 [s]_{\zeta}:= t_i, \quad \mbox{if}\quad t_i\leq s<t_{i+1}.$$
 \begin{defn} 
 The {\it iterated minmax solution operator} for the (H-J) equation with  respect to a subdivision
 $\zeta $ is defined as follows: for $0\leq s'<s \leq T$,  
 \[
 R_{H,\zeta }^{s',s}:=R_H^{t_j,s}\circ R_H^{t_{j-1},t_j}\circ\dots\circ R_H^{t_{i+1},t_{i+2}}\circ R_H^{s',t_{i+1}},\quad \text{where}\, t_j=[s]_{\zeta},t_i=[s']_{\zeta}, i\leq j.
  \] 
  When the Hamiltonian $H$ is fixed , we may abbreviate our notation
 $R_H^{s,t}$ as $R_s^t$, and the iterated minmax as
 \be
 \label{iterated}
 R_{s',\zeta }^s:= R_{t_j}^s\circ R_{t_{j-1}}^{t_j}\circ\dots\circ R_{t_{i+1}}^{t_{i+2}}\circ R_{s'}^{t_{i+1}},\quad \text{where}\, t_j=[s]_{\zeta},t_i=[s']_{\zeta}, i\leq j.
 \ee 
 indicated.
 \end{defn}
 
 Define the length of $\zeta$ by $|\zeta|:= \max_i|t_i-t_{i+1}|$.
 Suppose that $(\zeta_n)_n$ is a sequence of subdivisions of $[0,T]$ such that
 $|\zeta_n|$ tends to zero as $n$ goes to infinity, and let
 $\big(R_{0,\zeta_n}^sv(x)\big)_n$ be the corresponding sequence of iterated minmax solutions for an
 initial function $v\in C^{\Lip}(\R^d)$.
 
 \begin{lem}\label{246} The sequence of functions $u_n(s,x):=R_{0,\zeta_n}^s v(x)$
 is equi-Lipschitz and uniformly bounded for $(s,x)\in [0,T]\times K$.
 \end{lem}
 \begin{proof} By Proposition~\ref{prof}, one can verify that
 \beaa  &&\|\p(R_{0,\zeta_n}^s v)\|\leq \|\p v\|+ T\|\p_x H\|,\\
       &&|R_{0,\zeta_n}^s v-R_{0,\zeta_n}^t v|_K \leq |H|_{\mathcal{K}} |s-t|,\qquad s,t\in [0,T].\eeaa
 where $\mathcal{K}:=\{(t,x,y): t\in [0,T], x\in K, |y|\leq \|\p v\|+
 T\|\p_x H\|\}$.
 
 In particular, taking $t=0$, we  get
 \[|R_{0,\zeta_n}^{s}v|_K\leq |v|_K + T |H|_{\mathcal{K}},\qquad s\in
 [0,T]\] \end{proof}

 \begin{prop} 
 \label{suvis}
 For any sequence $(\zeta_n)_n$ of subdivisions of $[0,T]$ such that $|\zeta_n|\to 0$ as $n\to\infty$, and any compact $K\subset \R^d$, the sequence of iterated minmax $u_n(s,x):=R_{0,\zeta_n}^sv(x)$ has a subsequence converging uniformly on $[0,T]\times K$ to the viscosity solution of the (H-J) problem.
 \end{prop}
 
 \begin{proof} By Lemma \ref{246} and the Arzela-Ascoli theorem, the sequence of functions $u_n$ takes its values in a compact subset of $C^0([0,T]\times K)$, hence $(u_n)_n$ has a convergent subsequence $(u_{n_k})_k$. Denote by $\bar{R}_0^sv(x)$ its limit. We will prove that
  \be\label{semi}\bar{R}_0^{s} v(x)=\lim_{k\to\infty}
  R_{s',\zeta_{n_k}}^s\circ \bar{R}_0^{s'}v(x),\quad \forall\, 0\leq s'<s\leq T.\ee
 Let $\tilde{K}\supset K$ as defined in $4)$ of Proposition~\ref{prof}, applying the Arzela-Ascoli theorem to $\{u_{n_k}\}\subset C^0([0,T]\times \tilde{K})$ and extracting a subsequence of $n_k$ if neccessary, we may suppose that $u_{n_k}$ converges uniformly in $[0,T]\times \tilde{K}$.
 
 For simplicity of notation, we omit the subindex $k$ of $n_k$ in the following.

 Denote $[s]_n:=[s]_{\zeta_n}$. We first remark that
 \be\label{eqa}\bar{R}_{0}^sv(x)=\lim_{n\to\infty}R_{0,\zeta_n}^{[s]_n}v(x),\quad 0\leq s\leq T.\ee
 Indeed, \beaa |R_{0,\zeta_n}^s v(x)-R_{0,\zeta_n}^{[s]_n}v(x)|&=&|R_{[s]_n}^s
 \circ R_{0,\zeta_n}^{[s]_n}v(x)-R_{0,\zeta_n}^{[s]_n}v(x)|\\&\leq&
 |H|_{\mathcal{K}}(s-[s]_n)\leq |H|_{\mathcal{K}}|\zeta_n| \to 0,
 \quad n\to \infty.\eeaa

For any
 $\epsilon>0$, there exists $N$ large enough such that for any
 $i,j>N$,
 \[|R_{0,\zeta_i}^{[s]_i}v-R_{0,\zeta_j}^{[s]_j}v|_{\tilde{K}}< \epsilon,\quad \forall s\in [0,T]\] Hence \beaa 
 |R_{[s']_i,\zeta_i}^{[s]_i}\circ
 R_{0,\zeta_j}^{[s']_j}v-R_{0,\zeta_i}^{[s]_i}v|_K&=&|R_{[s']_i,\zeta_i}^{[s]_i}\circ
 R_{0,\zeta_j}^{[s']_j}v-R_{[s']_i,\zeta_i}^{[s]_i}\circ
 R_{0,\zeta_i}^{[s']_i}v|_K\\&\leq&
 |R_{0,\zeta_j}^{[s']_j}v-R_{0,\zeta_i}^{[s']_i}v|_{\tilde{K}}<\epsilon.\eeaa Let
 $j$ go to infinity, we get
 \[|R_{[s']_i,\zeta_i}^{[s]_i}\circ \bar{R}_0^{s'}v-R_{0,\zeta_i}^{[s]_i}v|_K<\ep,\quad i>N\]
 Thus the limit $\lim_{i\to\infty}R_{[s']_i,\zeta_i}^{[s]_i}\circ \bar{R}_0^{s'}v(x)$ exists and 
 \[\lim_{i\to\infty}R_{[s']_i,\zeta_i}^{[s]_i}\circ \bar{R}_0^{s'}v(x)=\bar{R}_0^sv(x),\quad x\in K.\]
 
  We conclude (\ref{semi}) by verifying the following, which
 is similar to (\ref{eqa}),
 \[\lim_{i\to \infty}R_{s',\zeta_i}^s\circ
 \b_0^{s'}v(x)=\lim_{i\to\infty}R_{[s']_i,\zeta_i}^{[s]_i}\circ
 \bar{R}_0^{s'}v(x).\]
 
 Now we show that the limit function $\bar{R}_0^sv(x)$ is the viscosity solution of the (H-J) problem. We first show that it is a
 viscosity subsolution. For any $(t,x)$, suppose $\psi$ a $C^2$
 function defined in a neighborhood of $(t,x)$, having bounded second derivative
 and such that $\psi(s,y)=:\psi_s(y)\geq
 \bar{R}_0^sv(y)$, with equality at $(t,x)$,
 \be\label{subsol1}\psi_t(x)=\bar{R}_0^tv(x)= \lim_{k\to
 \infty}R_{\tau,\zeta_{n_k}}^t\circ \bar{R}_0^{\tau}v(x)\leq \lim_{k\to
 \infty}R_{\tau,\zeta_{n_k}}^t\psi_{\tau}(x)=R_{\tau}^t\psi_{\tau}(x)\ee the last
 equality holds for
 $t-\tau$ small enough, where the characteristics originating
 from $d\psi_{\tau}$ do not intersect, hence the iterated minmax is
 nothing but the 1-step minmax which is the classical $C^2$
 solution. We conclude by applying the same argument in Proposition
 \ref{smgp}: 
  \be \label{hjc1} R_{\tau}^t\psi_{\tau}(x)=\psi_{\tau}(x)-\int_{\tau}^t H(s,x,\p_x
  R_{\tau}^s\psi_{\tau}(x))ds\ee
  
   Substracting (\ref{hjc1}) into (\ref{subsol1}), moving
  $\psi_t(x)$ to the right-hand side,  dividing both side by $t-\tau$ and lettting
  $\tau\to t$, we get
  \[0\leq -\p_t\psi_t(x)-H(t,x,\p_x \psi_t(x))\]
  from which we get a subsolution by definition. Similarly, we can
  prove that $R_0^tv(x)$ is a viscosity supersolution.
 \end{proof}
 
 For given $H$ and $v$, we say that {\it the limit of iterated minmax solutions exists} in $[s,t]$,
 if for any sequence of subdivision
 $\{\zeta_n\}_{n\in \N}$ of $[s,t]$ such that $|\zeta_n|\to 0$ as
 $n\to\infty$, the related sequence of iterated minmax solutions
 $\{R_{H,\zeta_n}^{s,\tau}v(x)\}_{n\in \N}$, $(\tau,x)\in [s,t]\times \R^d$  converges uniformly
 on compact subsets to a limit
 which is independent of the choice of subdivisions, then, without ambiguity, we denote
this limit also by $\bar{R}_H^{s,\tau}v(x)$. 

 We can now prove our main Theorem
 
 \begin{thm}\label{main} Suppose $H\in C_c^2([0,T]\times T^*\R^d)$ and $v\in C^{\Lip}(\R^d)$,
 then for the Cauchy problem of the Hamilton-Jacobi equation
 \[\left\{
     \begin{array}{ll}
       \p_t u + H(t,x, \p_x u) = 0, \quad t\in (0,T] \\
       u(x,0)= v(x),\quad x\in \R^d.
     \end{array}
   \right.\]
 the limit of iterated minmax solutions
  exists and coincides with the viscosity
 solution.
 \end{thm}
 
 \begin{proof} Let $K\subset \R^d$ and $\{\zeta_n\}_n$ is any sequence of subdivisions of $[0,T]$ such that $|\zeta_n|\to 0$ as $n\to \infty$. Denote $u_n(t,x):=R_{0,\zeta_n}^tv(x)$, and $u(t,x)$ the viscosity solution of the (H-J) problem. If $u_n$ does not converge uniformly on $[0,T]\times K$, there exists a $\ep>0$ and a subsequence $n_k$ such that 
$|u_{n_k}-u|_{[0,T]\times K}>\ep$. Note that $\{\zeta_{n_k}\}$ is itself a sequence of subdivisions, this contradicts with  Proposition \ref{suvis}.
 \end{proof}

\appendix
\section{Lipschitz critical point theory} 
\label{appA}
{\renewcommand\thesection{A}
 The notion of Clarke's generalized derivative for Lipschitz functions has been well adapted in the Hamilton-Jacobi theory, such as the works in \cite{FT,Siconofli, PJ,arnaud}. We will give a brief review on Lipschitz critical point theory, extracted from Appendix~A in \cite{these}.  Let us consider a real locally Lipschitz function $f$ on
$X:=\R^k$.

The \emph{Clarke generalized derivative} $\p f(a)$ of $f$ at $a\in X$ is the convex subset $\p f(a)$ of $ X^*=T_x^* X$ defined as follows:  by Rademacher's theorem, the set $\dom(df)$ of differentiability points of $f$ is dense in $ X$\/; if $df:=\big\{\big(x,df(x)\big):x\in\dom(df)\big\}$\/, we let
\[
\p f(a):= \conv\{y\in X^*:(a,y)\in\overline{df}\},
\]
where $\hbox{co}$ stands for the convex hull; in other words, $\p f(a)$ is the convex hull of the set of limits of convergent sequences $df(x_n)$ with $\lim x_n=a$\/. As $|df(x)|$ is bounded by the local Lipschitz constant of $f$ for $x$ close to $a$\/, every sequence $df(x_n)$ with $\lim x_n=a$ is bounded and therefore has a convergent subsequence, implying
\[
\forall a\in X\quad\p f(a)\neq\emptyset\/;
\]
moreover, $\p f(a)$ \emph{is compact}, being the convex hull of a compact subset. The subset
\[
\p f:=\{(x,y)|y\in \p f(x),x\in  X\}
\]
is a generalized version of the {\it enlarged pseudograph} defined for semi-concave
functions in \cite{Ma}, where the  {\it pseudograph} is $df$\/.
In simple one-dimensional cases\/, it is obtained by adding a vertical segment to
$df$ where $f$ is not differentiable:

\begin{rem}
\label{remB2}
The set $\p f(x)$ consists of a single point if and only if $f$ is ``$C^1$ at $x$
with respect to the set where it is differentiable''.
\end{rem}

\begin{prop}
\label{lemB11}
The set-valued function $x\mapsto \p f(x)$ is \emph{upper semi-continuous:} for every convergent sequence $(x_n,y_n)\to (x,y)$ with $y_n\in \p f(x_n)$, one has $y\in \p f(x)$\/. In other words, $\p f$ is closed.
\end{prop}

\begin{defn}
\label{crit}
A point $x\in X$ is called a \emph{critical point} of $f$ if $0\in \p f(x)$; the number $f(x)$ is then called a \emph{critical value} of $f$\/. By Proposition~\ref{lemB11}, the \emph{critical set} $\hbox{\rm Crit}(f)$ of $f$\/, consisting of its critical points, is closed in $X$\/.

Setting
\[
\lambda(x):=\min_{w\in\p f(x)}|w|_{X^*}\/,
\]
we say that $f$ satisfies the \emph{Palais-Smale condition} (P.S.) if every sequence $(x_n)$ along
which $f(x_n)$ is bounded, and such that $\lambda(x_n)$ goes to $0$, possesses a convergent subsequence---whose limit is a critical point of $f$ by Proposition~\ref{lemB11}, as there is a sequence $y_n\in\p f(x_n)$ converging to $0$\/.
\end{defn}

\begin{ex}
\label{exB7}
The P.S. condition is satisfied when $\Lip(f-Q)<\infty$ for some \emph{nondegenerate} quadratic form $Q$ on $X$\/; moreover, in that case, $\hbox{\rm Crit}(f)$ is \emph{compact.}
\end{ex}

\begin{proof}
Indeed, if $\psi:=f-Q$\/, each subset $\p f(x)= \p\psi(x)+ d Q(x)$ consists of vectors whose norm is at least $|d Q(x)|-\Lip(\psi)$\/, hence $\lambda(x)\geq|d Q(x)|-\Lip(\psi)$\/, which tends to $+\infty$ when $|x|\to\infty$\/; therefore, there exists $R>0$ such that every sequence $(x_n)$ with $\lim\lambda(x_n)=0$ satisfies $|x_n|\leq R$ for all large enough $n$\/, implying both the P.S. condition and the compactness of $\hbox{\rm Crit}(f)$\/.
\end{proof}

\begin{thm}[Deformation Lemma I]
\label{def1}
Suppose $f$ satisfies the P.S. condition and let $f^c:=\{x|f(x)\leq c\}$ for each $c\in\R$\/. If $c$ is not a critical value of $f$\/, then there exist $\ep>0$ and a bounded smooth vector field $V$ on $X$ equal to $0$ off $f^{c+2\ep}\smallsetminus f^{c-2\ep}$\/, and whose flow $\vp_V^t$ satisfies $\vp_V^1(f^{c+\ep})\subset f^{c-\ep}$\/.
\end{thm}

\begin{thm}[Deformation Lemma II]
\label{def2}
Suppose  $f$ satisfies the P.S. condition. If $c\in \R$ is a critical value of $f$ and $N$ \emph{any} neighbourhood of $K_c:=\hbox{\rm Crit}(f)\cap f^{-1}(c)$, then there exist $\ep>0$ and a bounded smooth vector field $V$ on $X$ equal to $0$ off $f^{c+2\ep}\smallsetminus f^{c-2\ep}$\/, whose flow $\vp_V^t$ satisfies $\vp_V^1(f^{c+\ep}\smallsetminus N)\subset f^{c-\ep}$\/.
\end{thm}
 
\begin{lem}[Chain rule] \label{B12}If $f:X\to \R$ is a Lipschitz function,
$F:X\to X$ a $C^1$ diffeomorphism, then
\[\p (f\circ F)(x)=\p f(F(x))\circ d F(x):=\{ dF(x)(\xi),\,\xi\in
\p f(F(x))\}.\]
\end{lem}

\begin{lem} \label{B16}If $f,g:X\to \R$ are Lipschitz functions, then
\[\p(fg)(x)\subset f(x)\p g(x)+g(x)\p f(x).\]
\end{lem} 
 }
 
\section{An example} \label{rarefaction}

Let us look at a simple example where the minmax and the viscosity solution differ to see how the iterated minmax converges to the latter. 
\label{rare} Consider the (H-J) problem  for a conservation law in one space dimension, i.e. $H$ depends only on $p$. In this case, the minmax\footnote{By the {\it finite propagation speed} property, the minmax exists when $v$ is globally Lispchitz and $H(p)$ is $C^1$.} is given by
\[R_0^tv(x)=\inf\max S_t(x;x_0,y_0):=\inf\max (v(x_0)+xy_0-tH(y_0)-x_0y_0)\]
and the geometric solution and wave front are given by
\beaa L_t&=&\vp^t(\p v)=\{(x,y_0)|x=x_0+tH'(y_0), y_0\in \p v(x_0)\},\\ \mathcal{F}^t&=&\mathcal{F}^t(v)=\{(x,S_t(x;x_0,y_0))|x=x_0+tH'(y_0), y_0\in \p v(x_0)\}.\eeaa
For an example, we take
\[H(p)=-p^3+p^2+p\]
and the initial function $v(x)$ globally Lipschitz and in a neighbourhood of $x=0$, \[v(x)=\left\{
     \begin{array}{ll} -x(x-1),\quad x\leq 0\\x(x-1),\quad\quad x\geq 0\end{array}
        \right.\]
See Figure \ref{case11}, where $p_0^{\pm}=v'(0\pm)$.
\input{case11.TpX}

For $t>s>0$ small, the geometric solution and
wave fronts are depicted as in Figure \ref{m1},\ref{L1},\ref{m2},\ref{L2} below. If $A$ is a subset of $\R$, we denote
\[\mathcal{F}^t_A:=\mathcal{F}^t|_{\{x_0\in A\}}=\{(x,S_t(x;x_0,y_0))\in \mathcal{F}^t, x_0\in A\}.\]

\begin{figure}
\begin{minipage}[t]{0.5\linewidth}
\begin{flushleft} \large
\includegraphics[width=2.2in]{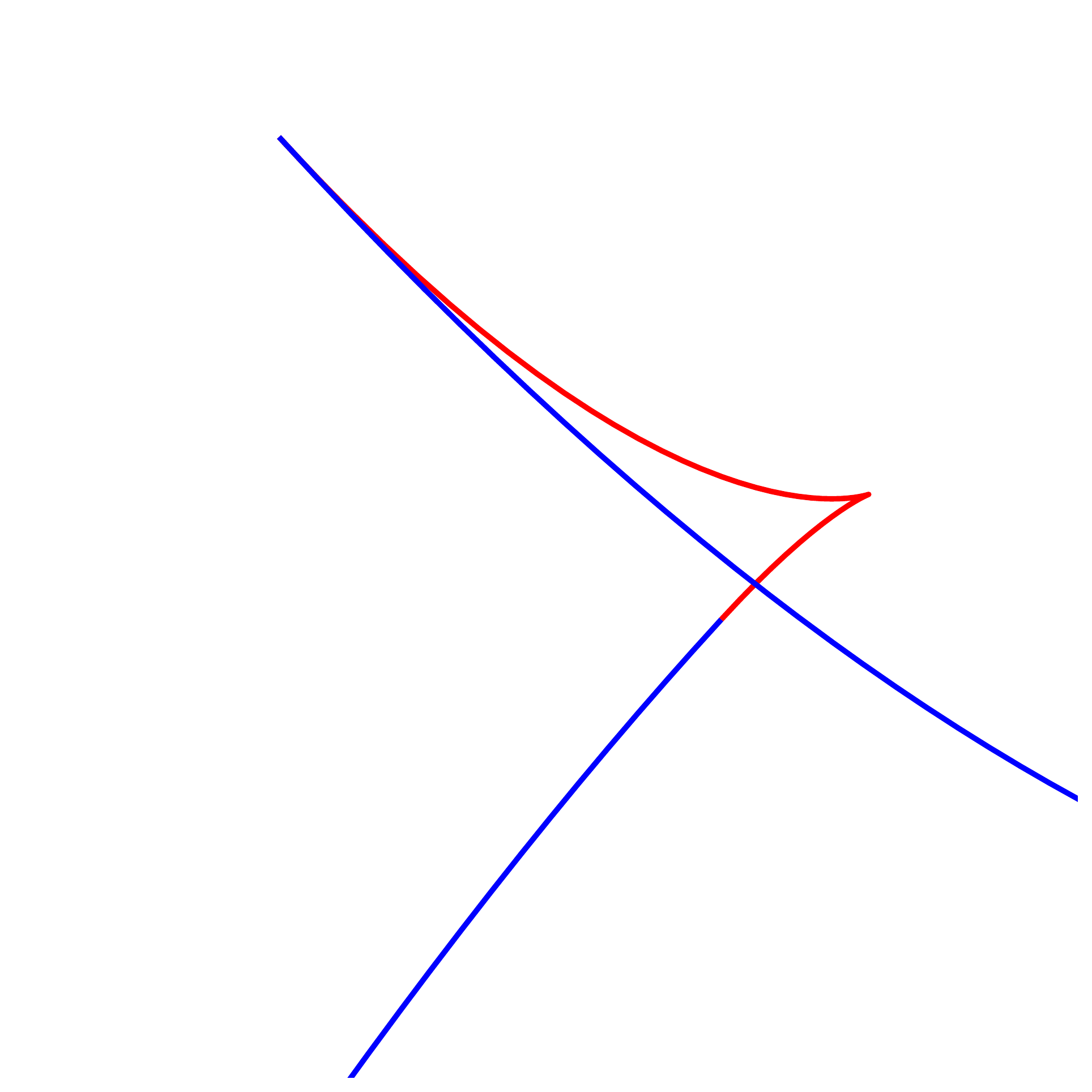}
\caption{$\mathcal{F}^t(v)$ }\label{m1}
\end{flushleft}
\end{minipage}%
\begin{minipage}[t]{0.5\linewidth}
\begin{flushright} \large
\includegraphics[width=2.0in]{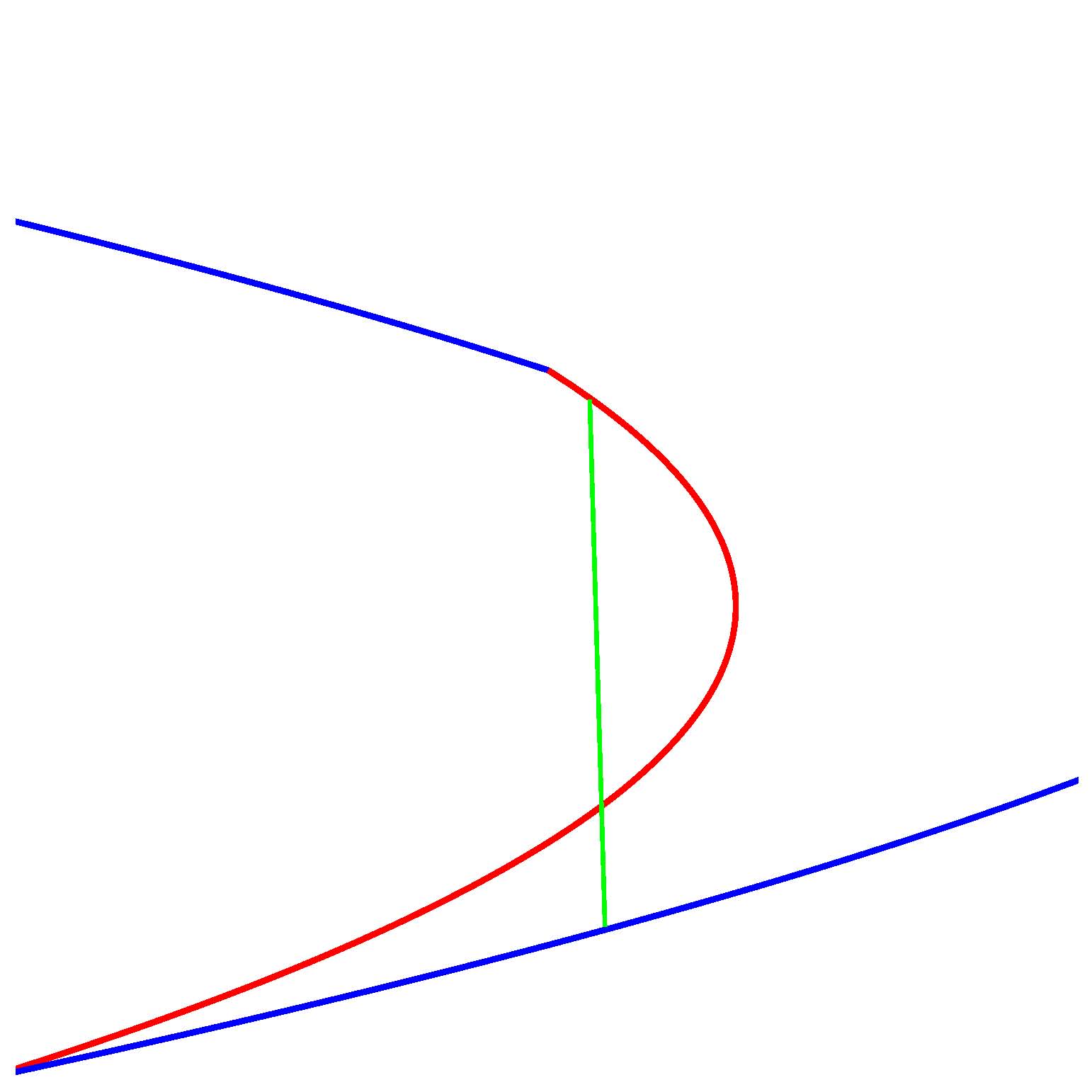}
\caption{$\vp^t(\p v)$}\label{L1}
\end{flushright}

\end{minipage}
\end{figure}

\begin{figure}
\begin{minipage}[t]{0.5\linewidth}
\begin{flushleft} \large
\includegraphics[width=2.2in]{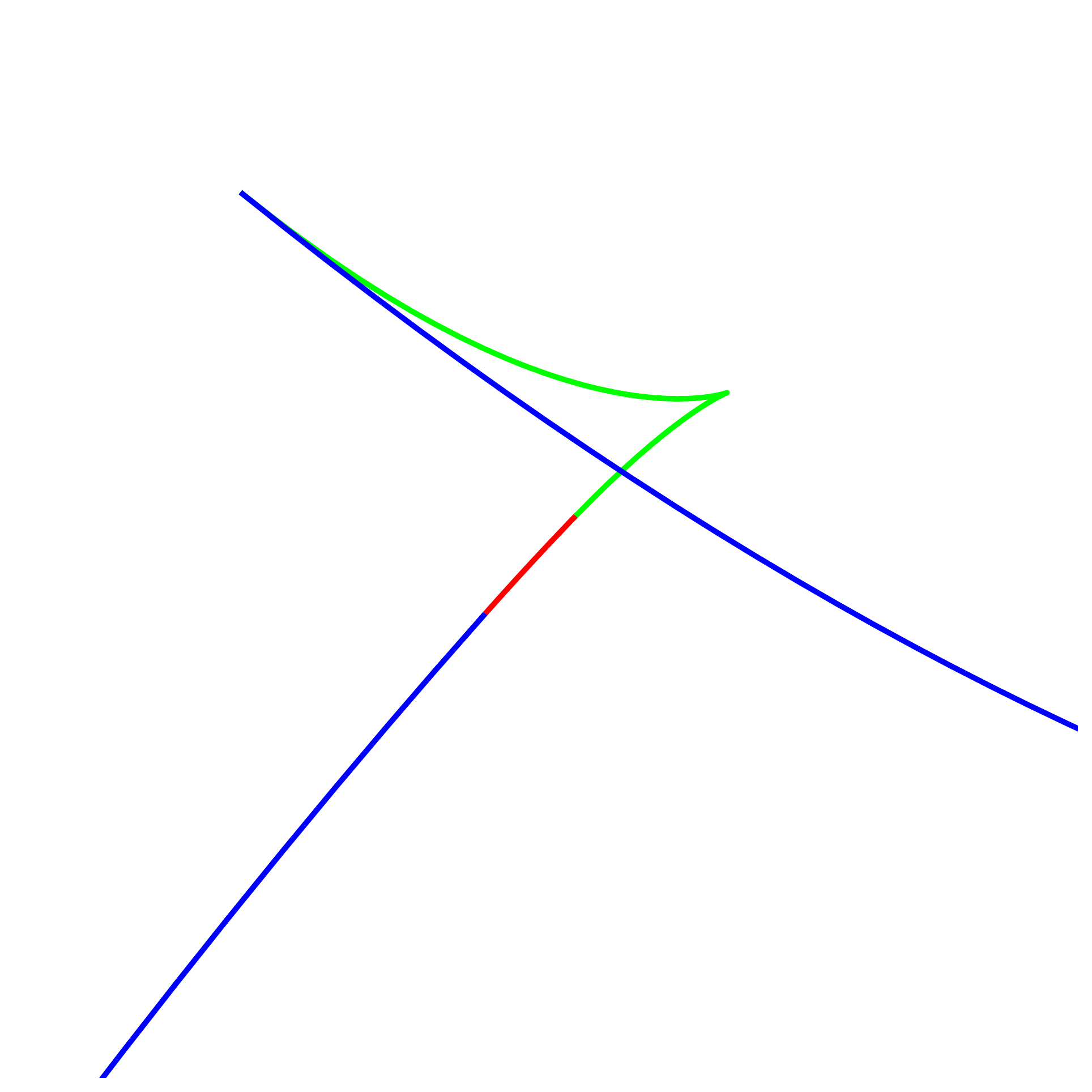}
\caption{$\mathcal{F}^{t-s}(R_0^sv)$ }\label{m2}
\end{flushleft}
\end{minipage}%
\begin{minipage}[t]{0.5\linewidth}
\begin{flushright} \large
\includegraphics[width=2.2in]{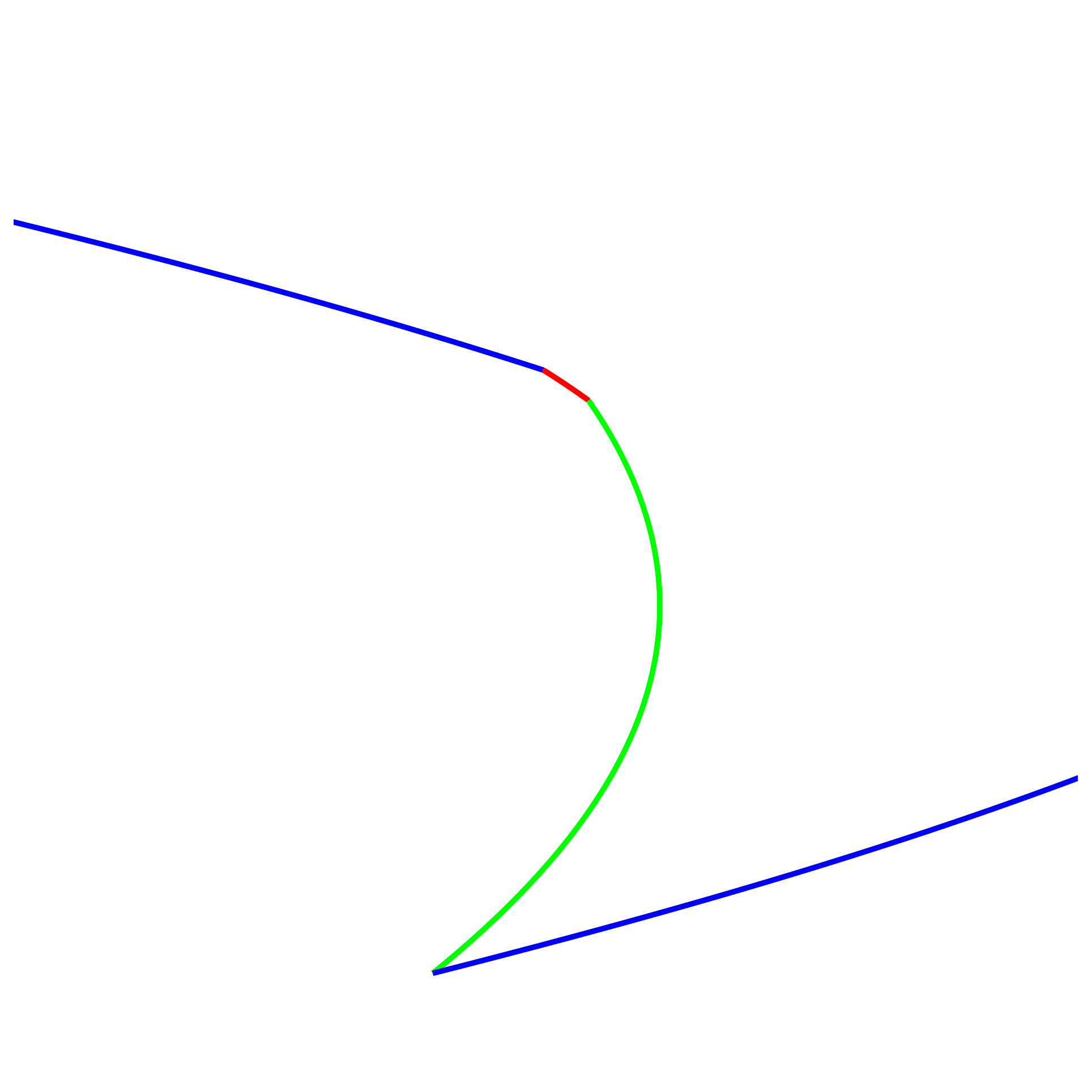}
\caption{$\vp^{t-s}(\p R_0^sv)$}\label{L2}
\end{flushright}

\end{minipage}
\end{figure}

In the wave front $\F1^t(v)$ (ref. Figure \ref{m1}), the two branches in blue are genuine branches $\mathcal{F}^t_{\{x_0>0\}}(v)$ and $\mathcal{F}^t_{\{x_0<0\}}(v)$, and the curve in red is $\F1_{\{0\}}^t(v)$. Being a continuous section, the 1-step minmax
$R_0^tv(x)$ has no choice but being the minimum in the wave front: it contains a piece (in red) generated by the ``vertical segment'' $\p v(0)$ in Clarke's generalized derivatives which describe the singularity of $v$ at $x=0$. This phenomenon reproduces when taking iterated minmax. For 2-step minmax (ref. Figure \ref{m2}): let $x_s$ denote the singularity of the derivative of $R_0^sv$, then $R_s^t\circ R_0^sv$ contains a piece (in green) from $\p R_0^sv(x_s)$.

The presence at each step of these new pieces implies that the minmax does not form a semi-group and that the viscosity solution is not contained in the geometric solution $\mathcal{F}^t(v)$, see Figure \ref{front3}: the part not belonging to the geometric solution comes, through the limiting process,  from the singularities of the derivative which appear at each step of the iteration.
\input{front3.TpX}
\section*{References}
\bibliographystyle{alpha}
\bibliography{bibliothese} 
 
\end{document}